\documentclass[11pt]{amsart}
\usepackage[utf8]{inputenc}
\usepackage{amsmath}
\usepackage{amsthm}
\usepackage{amsfonts}
\usepackage{amssymb}
\usepackage{enumitem}
\usepackage{esint}
\usepackage{color}

\title{Locally uniform domains and extension of bmo functions}
\author[Butaev]{Almaz Butaev}
	\address{(A.B.) Department of Mathematical Sciences, P.O. Box 210025, University of Cincinnati, Cincinnati, OH 45221–0025, U.S.A}
	\email{butaevaz@ucmail.uc.edu}
	
	\author[Dafni]{Galia Dafni}
	\address{(G.D.) Concordia University, Department of Mathematics and Statistics, Montr\'{e}al, QC H3G 1M8, Canada}
	\email{galia.dafni@concordia.ca}

\thanks{G.D. was partially supported by the Natural Sciences and Engineering Research Council (NSERC) of Canada, and the Centre de recherches math\'e{}matiques (CRM)}

\date{}

\newtheorem{manualtheoreminner}{Theorem}
\newenvironment{theoremA}[1]{%
\renewcommand\themanualtheoreminner{#1}%
\manualtheoreminner
}{\endmanualtheoreminner}

\newcommand{\R}{\mathbb{R}}
\newcommand{\N}{\mathbb{N}}
\newcommand{\Rn}{{\R}^n}
\newcommand{\ed}{{(\epsilon,\delta)}}
\newcommand{\dist}{\textup{dist}}
\newcommand{\diam}{\textup{diam}}

\newcommand{\led}{\lambda_{\epsilon, \delta}}
\newcommand{\Tlam}{{T_\lambda}}

\newcommand{\cD}{\mathfrak{D}}

\newcommand{\BMOo}{{\BMO(\Omega)}}
\newcommand{\Loneloc}{{L^1_{\text{loc}}}}

\newcommand{\bOmega}{{\partial \Omega}}
\newcommand{\dom}{{d_\Omega}}
\newcommand{\dOmega}{{d_\Omega}}

\newcommand{\ko}{{k_\Omega}}
\newcommand{\kOmega}{{k_\Omega}}

\newcommand{\jo}{{j_\Omega}}
\newcommand{\jOmega}{{j_\Omega}}

\newcommand{\Omegal}{\mathring{\Omega}_{\lambda}}

\newcommand{\BMO}{{\rm BMO}}
\newcommand{\bmo}{{\rm bmo}}
\newcommand{\VMO}{{\rm VMO}}

\newcommand{\bmol}{{\bmo_\lambda}}

\newcommand{\bmolo}{{\bmo_\lambda(\Omega)}}

\newcommand{\cN}{\mathcal{N}}
\newcommand{\cA}{\mathcal{A}}

\newcommand{\ra}{\rightarrow}

\newtheorem{theorem}{Theorem}[section]
\newtheorem{lem}[theorem]{Lemma}
\newtheorem{definition}[theorem]{Definition}
\newtheorem{prop}[theorem]{Proposition}
\newtheorem{cor}[theorem]{Corollary}
\newtheorem{remark}[theorem]{Remark}

\begin{document}
\maketitle
\begin{abstract}We prove that for a domain $\Omega \subset \Rn$, being $\ed$ in the sense of Jones is equivalent to being an extension domain for $\bmo$, the nonhonomogeneous version of the space of function of bounded mean oscillation on $\Omega$.  In the process we identify $\ed$-domains with a local version of uniform domains, defined by requiring the presence of length cigars between nearby points.  Our results show that the definition of $\bmo(\Omega)$ is closely tied to the geometry of the domain.
\end{abstract}

\section{Introduction} 
Let $\Omega$ be a domain (open connected set) in $\Rn$ and consider a space $X$ of functions on $\Omega$. We are interested in the relationship between the geometry of the domain $\Omega$ and the existence of a bounded extension operator $T:X(\Omega) \to X(\Rn)$. Can we characterize the domain in terms of such extensions, namely give conditions on $\Omega$ which are necessary and sufficient for the existence of the operator $T$?  If such a map exists,  we say that the domain is an {\em extension domain} for the given function space $X$. We refer to \cite{brudnyis1} and \cite{brudnyis2} for a comprehensive account of  extension problems for different function spaces $X$.  

Although extension domains have been widely studied, the problem of their complete characterization is still open in general, even for the classical Sobolev spaces $W^{s,p}$ (see however \cite{shvartsman_zobin}). When $X=\BMO$,  a complete description of the extension domains has been given by Jones \cite{jones1}. The space $\BMO(\Omega)$ consists of functions $f$ which are locally integrable in $\Omega$ and have bounded mean oscillation in $\Omega$, that is, 
\begin{equation}
\label{def-BMOo}
\|f\|_\BMOo:= \sup_{Q\subset \Omega} \fint_Q |f(x) - f_Q| dx < \infty,
\end{equation}  
where here and in what follows $Q$ denotes a cube with sides parallel to the axes, $|Q|$ is its measure and $f_Q = \fint_Q f :=  \frac{1}{|Q|} \int_Q f dx$ is the mean of $f$ over $Q$.  The seminorm $\|f\|_\BMOo$ defines a norm modulo constants.

Jones gave a geometric condition on the domain which was necessary and sufficient for the existence of a bounded extension map from $\BMO(\Omega)$ to $\BMO(\Rn)$.  This condition was shown in \cite{gehring_osgood} to be equivalent to $\Omega$ being a {\em uniform domain}, previously defined in \cite{martio_sarvas} (see also \cite{vaisala}).  In recent work of the authors \cite{ButaevDafni}, uniform domains are identified as extension domains for the space $\VMO$ of functions of vanishing mean oscillation.

In \cite{jones2}, Jones introduced local versions of uniform domains called $\ed$-domains (uniform domains are $(\epsilon, \infty)$). 
He showed that such domains are extension domains for the Sobolev spaces $W^{s,p}$, $1 \leq p \leq \infty$, $s \in \N$ and this result is sharp in dimension $2$: if a finitely connected open set is an extension domain for all Sobolev spaces, then it is an  $(\epsilon, \delta)$ domain (this fails in higher dimensions - for recent results see \cite{Rajala}). Christ in \cite{christ} extended Jones' results to certain spaces of fractional smoothness which had been simultaneously studied by DeVore and Sharpley in \cite{devore_sharpley} and later in \cite{devore_sharpley2}; there followed many further extensions and variations - for some recent examples see \cite{BreitCianchi,koskela}.

From the results of Jones and Christ it follows that uniform domains are extension domains for the homogeneous Sobolev spaces, and $\ed$-domains for the nonhomogeneous ones.  BMO can be considered as the zero smoothness endpoint of the homogeneous case.  The natural question then arises in the nonhomogeneous case:  does the extension result for $\ed$-domains hold with zero smoothness, and what is the corresponding nonhomogeneous BMO space? 

Since the converse direction for Sobolev extensions only holds when the dimension is $2$, one can also ask whether  there is a space for which  the existence of an extension map implies  that $\Omega$ is an $\ed$-domain in \textit{any} dimension.

We answer both these questions by identifying $\ed$-domains as the extension domains for $\bmo$, known alternatively as {\em local, localized} or {\em nonhomogeneous} BMO.  The space $\bmo(\Rn)$ was introduced by Goldberg \cite{goldberg} and consists of locally integrable functions $f$ satisfying
$$\|f\|_{\bmo(\Rn)} : = \sup_{\ell(Q) < 1} \fint_Q |f(x) - f_Q| dx  + \sup_{\ell(Q) \geq 1} |f|_Q < \infty,$$
where the supremum is taken over all cubes $Q \subset \Rn$, and $\ell(Q)$ denotes the sidelength of $Q$.
Replacing Goldberg's choice of the constant $1$ in the definition  by a finite constant $\lambda$ gives the same space with equivalent norms, and $\BMO(\Rn)$ can be considered as the case $\lambda = \infty$.   While the norm in $\bmo(\Rn)$ is not taken modulo constants, considered as sets of functions, $\bmo(\Rn)$ is a proper subset of $\BMO(\Rn)$; for example, $\log |x| \in \BMO(\Rn)\setminus\bmo(\Rn)$.  

On a bounded domain $\Omega$,  $\bmo(\Omega)$ is just $\BMO(\Omega)$ equipped with a nonhomogeneous norm, adding to \eqref{def-BMOo} the supremum of the averages of the function over large cubes or balls contained in $\Omega$, or alternatively the $L^1(\Omega)$ norm of the function, and corresponds to the zero smoothness case in the scale of nonhomogeneous Morrey-Campanato spaces (see  Triebel  \cite{triebel}, Section 1.7.2).  However, when $\Omega$ is unbounded, the choice of constant, or {\em scale}, in the definition of $\bmo(\Omega)$ is important.  

Fix $\lambda > 0$.  We say  $f \in \bmolo$ if $f$ is integrable on every cube $Q \subset \Omega$ and 
\begin{equation}
\label{def-bmolo}\|f\|_\bmolo := \sup_{Q\subset \Omega, \ell(Q) < \lambda} \fint_Q |f(x) - f_Q| dx  + \sup_{Q\subset \Omega, \ell(Q) \geq \lambda} |f|_Q < \infty.
\end{equation}
If the domain does not contain any cube of sidelength $\lambda$ or larger, we have $\bmolo = \BMOo$ with the same norms, namely we are back in the homogeneous case.  Thus we need to choose $\lambda$ sufficiently small for a cube of sidelength $\lambda$ to be contained in $\Omega$, but this is not enough.  For example, consider the Lipschitz domain $\Omega \subset \R^2$ lying between the $x$-axis and the curve $y = \max(1, 1-x)$. Clearly $\Omega$ contains arbitrarily large cubes,  but  choosing $\lambda$ too large, say $\lambda > 1$, results in the existence of functions  belonging to $\bmolo$, such as $f(x,y) = \max(x,0)$, but not having an extension to $\bmo(\R^2)$. 
Thus the correct scale for the definition of $\bmo$ on a domain is intimately connected with the geometry.  

We now state the main result of this paper.

\begin{theorem}
\label{thm_ed_is_extension}		
If $\Omega$ is an $\ed$-domain then there is a positive constant $\lambda_0$ such that for each $\lambda\leq \lambda_0$,  
		we have a bounded linear extension operator $\Tlam :\bmolo \to \bmo(\Rn)$.  Conversely, if for some $\lambda > 0$ there is a bounded extension operator $\Tlam :\bmolo \to \bmo(\Rn)$,  then $\Omega$ is an $\ed$-domain for some $\epsilon, \delta$ depending on $\lambda$.	
\end{theorem}

In particular, if $\Omega$ is an $\ed$-domain, we see that for each $\lambda\leq \lambda_0$, the set $\bmolo$ consists precisely of restrictions to $\Omega$ of elements of $\bmo(\Rn)$, hence all these sets are the same and we can define $\bmo(\Omega)$ to be $\bmolo$ with $\lambda = \lambda_0$.

Before proving the theorem, we identify $\ed$-domains with \textit{locally uniform} domains, used in the sense of Herron and Koskela \cite{herron} - see Section~\ref{sec-locallyuniform}.  The key ingredient consists of enhanced localized versions of two theorems of Gehring and Osgood \cite{gehring_osgood}, relating the existence of a certain curve, or {\em cigar}, between two points $x,y \in \Omega$, and the comparability of the {\em distance-ratio metric} $\jo(x,y)$ and the {\em quasihyperbolic metric} $\ko(x,y)$, introduced by Gehring and Palka \cite{Gehring_Palka}.  We define these terms and prove the Gehring-Osgood-type theorems  in Section~\ref{sec-uniformdomain}.   

Other variants of these conditions can be found in \cite{martel}, where the notion of {\em semi-uniform} is introduced, and in \cite{mitreas}, where the $\ed$ condition is further localized.  In \cite{LukeRogers}, the term \textit{locally uniform} is added to what is just Jones' definition of $\ed$-domains.  On the other hand,  the condition that is called {\em locally uniform}  in \cite{mitreas} is stated in terms of what are known as {\em distance cigars} (as opposed to the {\em length cigars} in Definition~\ref{defin_ed_GO}).  This condition, which we called {\em locally distance uniform}, is immediately equivalent to $\ed$ - see \cite{vaisala}, Section 2.4.   The main difficulty is replacing length cigars by distance cigars in the hypothesis of Theorem 1 of \cite{gehring_osgood} (Theorem~\ref{G-O_TH1} below).

Once the equivalence of the definitions is established, we prove, in Section~\ref{sec-necessity}, the necessity of the $\ed$ condition in Theorem~\ref{thm_ed_is_extension}.  This is stated as Theorem~\ref{thm_bmo_ext_is_uniform}, under the weaker hypothesis that there is a bounded extension operator from $\bmolo$ to $\BMO(\Rn)$.   Note that while the necessity in Jones \cite{jones1} is essentially a consequence of the fact that  $\ko \in \BMO(\Omega)$ with bounded norm, the situation is more complicated in the local case as one needs to control the nonhomogeneous norm.  This is done by comparing $\ko$ with the distance to a the set point of $\Omegal$ lying away from the boundary. 

The sufficiency  in Theorem~\ref{thm_ed_is_extension}, Theorem~\ref{thm_extension}, is proved in Section~\ref{sec-extension}. The extension to $\Omega^c$ is defined as in \cite{jones1} near the boundary, but is set to zero far away from the boundary.  In order to control this jump in the $\bmo(\Rn)$ norm, one needs a logarithmic (in sidelength) growth bound on the averages of functions in $\bmolo$ on cubes (see Proposition~\ref{prop-ave-bound}).  While such a bound is standard in $\bmo(\Rn)$, it is the locally uniform condition that gives it to us for functions in $\bmolo$, for sufficiently small $\lambda$. As pointed out above, this may fail even on a Lipschitz domain if the choice of $\lambda$ is too large.

\section{Uniform domains}
\label{sec-uniformdomain}

In what follows, $\Omega$ will always denote a domain (open and connected set) in $\Rn$.  We denote by $d_\Omega(x)$ the distance to the boundary, from either inside or outside the domain:
$$\dom(x) := \dist(x,\bOmega).$$
For a rectifiable curve $\gamma$, we denote its arclength by $s(\gamma)$ or by $|\gamma|$.  If $p, q$ are two points along the curve $\gamma$, we denote by $\gamma(p,q)$ or $\gamma_{p,q}$ the portion of the curve between $p$ and $q$.  We will use the notation $B(x,r)$ to denote a ball of center $x$ and radius $r$.  Unless otherwise stated, balls and cubes are assumed to be closed, with $Q^\circ$ denoting the interior of $Q$.

Uniform domains were defined by Martio and Sarvas in \cite{martio_sarvas} as follows
\begin{definition}
\label{defin_ed_GO}
Given $a,b\geq 1$, we say that domain $\Omega$ is $(a,b)$-uniform if any $x,y\in \Omega$ can be connected by a rectifiable curve $\gamma = \gamma(x,y)\subset \Omega$ satisfying a quasiconvexity condition with constant $a$:
\begin{equation}
	\label{defin_ed_GO_a}
	s(\gamma(x,y)) \leq a |x - y|, 
\end{equation}
and a John condition with constant $b$: for all $z\in \gamma(x,y)$
\begin{equation}
	\label{defin_ed_GO_b}
	\min(s(\gamma(x, z)),s(\gamma(z,y))) \leq  b\dom(z).
\end{equation}

We say that $\Omega$ is a uniform domain if it is $(a,b)$-uniform for some choice of $a,b\geq 1$. 
\end{definition}

\subsection{Quasihyperbolic metric and Gehring-Osgood theorems}	
On any domain $\Omega\subset \Rn$ the \textit{quasihyperbolic metric} is defined as
\begin{equation}
\label{def-ko}
\ko(x,y) := \inf_{\gamma(x,y)\subset \Omega} \int_{\gamma(x,y)} \frac{ds}{\dist(z,\bOmega)}, \qquad x,y\in \Omega,
\end{equation}
where the infimum is taken over all rectifiable curves  connecting $x$ and $y$, and $ds$ is arclength. Gehring and Osgood showed (see Lemma 1 in  \cite{gehring_osgood}) that given any two points $x, y \in \Omega$, there is always a curve minimizing $\ko$, i.e. a rectifiable curve $\gamma = \gamma(x,y)$ for which
$$\int_{\gamma} \frac{ds}{\dist(z,\bOmega)} = \ko(x,y).$$ Any such curve is called a  {\em quasihyperbolic geodesic} between $x$ and $y$.  

Gehring and Palka in \cite{Gehring_Palka} considered another metric defined on an arbitrary domain $\Omega$, which they called the  \textit{distance-ratio metric} $\jOmega$, defined by 
\begin{equation}
\label{def-jo}
\jo(x,y):= \frac{1}{2}\log\left[\left(1+\frac{|x-y|}{\dom(x)}\right)\left(1+\frac{|x-y|}{\dom(y)}\right)\right], \qquad x,y \in \Omega,
\end{equation}
and established the estimates 
\begin{equation}\label{G-P1}
\left|\log\frac{\dOmega(x)}{\dOmega(y)}  \right| \leq \kOmega(x,y),	\qquad x,y\in \Omega,
\end{equation}
and
\begin{equation}\label{G-P2}
\max \left\{ \log\left(1+ \frac{|x-y|}{\dOmega(x)}\right), \; \log\left(1+ \frac{|x-y|}{\dOmega(y)}  \right) \right\}\leq \kOmega(x,y), \qquad \qquad x,y \in \Omega.
\end{equation}
In other words, $\jo(x,y)\leq \ko(x,y)$ for any $x,y\in \Omega$.

The following two theorems imply that a reverse inequality holds if and only if $\Omega$ is uniform. 
\begin{theoremA}{A}[Gehring-Osgood] \label{G-O_TH2}
Let $x,y\in \Omega$ and suppose $\gamma(x,y)$ is a quasihyperbolic geodesic. If there exists $C\geq 1$ such that for all $z,w\in \gamma(x,y)$,
\begin{equation} \label{G-O_TH2_cond}
	\ko(z,w)\leq C[1+\jo(z,w)],
\end{equation}
then $\gamma(x,y)$ satisfies \eqref{defin_ed_GO_a} and \eqref{defin_ed_GO_b} with 
\[
a\leq 48 C^2 e^{48C^2} \quad \text{ and } \quad b\leq 24C^2.
\]
\end{theoremA}

\begin{theoremA}{B}[Gehring-Osgood] \label{G-O_TH1}
Let $x,y\in \Omega$. If there exists $\gamma(x,y)\subset \Omega$ such that \eqref{defin_ed_GO_a} and \eqref{defin_ed_GO_b} hold with some $a,b\geq 1$, then 
\[
\ko(x,y)\leq C[1+\jo(x,y)]
\]
with $C\leq 2(b+b\log a + 1)$.
\end{theoremA}

\subsection{Localized Gehring-Osgood theorems} 
\subsubsection{Refinement of Theorem \ref{G-O_TH2} }
Theorem \ref{G-O_TH2} ensures that a quasihyperbolic geodesic curve $\gamma_{x,y}$ satisfies both \eqref{defin_ed_GO_a} and \eqref{defin_ed_GO_b} if the quasihyperbolic metric is controlled by the distance-ratio metric uniformly along this geodesic. 
We will relax the hypothesis by requiring \eqref{G-O_TH2_cond} to hold only locally (i.e.\ only for $\delta$-close $z,w$). The theorem below shows that such a refined formulation is indeed possible provided that the endpoints $x,y$ themselves are close enough. This additional condition on the endpoints is necessary. It can be verified e.g. by considering an infinite strip $\Omega = \{(x_1,x_2): x_1\in \R, |x_2|<1\}$.

We will also show that it suffices to control the quasihyperbolic metric by any  expression of the form 
\[
\phi\left(\frac{|u-v|}{\dOmega(u)}\right)+\phi\left(\frac{|u-v|}{\dOmega(v)}\right),
\]
where $\phi$ is a positive increasing function growing sub-linearly. The fact that $\log$ can be replaced by any sub-linear function in the characterization of uniform domains is known - see e.g.\ \cite{vaisala2} Theorems 6.16, 6.17 or \cite{herron2}, Theorem 3.1.  

\begin{theorem} \label{thm_GO-2_our version}
Let $\delta>0$ and $c\geq 1$ be fixed. Let $\phi$ be an increasing non-negative function on $[0,\infty)$ such that there are numbers $d_\phi\geq c_\phi\geq c\geq 2$ with 
\begin{equation}
	\label{tempC2_0}
	1+\phi(t)\leq \frac{t}{2c} \qquad \text{ for all } t\geq c_\phi 
\end{equation} 
and 
\begin{equation}
	\label{tempC2_00}
	1+\phi(t)\leq \frac{t}{c_\phi e^{8 c_\phi}} \qquad \text{ for all } t\geq d_\phi. 
\end{equation}
Suppose that $\gamma_{x,y}$ is a quasihyperbolic geodesic satisfying the following properties: 
\begin{itemize}
	\item For all $u,v\in \gamma_{x,y}$ with $|u-v|<\delta$,
	\begin{equation} \label{lem1_prop1}
		\kOmega(u,v) \leq c\left [1+\frac{1}{2}\left(\phi\left(\frac{|u-v|}{\dOmega(u)}\right)+\phi\left(\frac{|u-v|}{\dOmega(v)}\right)\right)\right]. 
	\end{equation} 
	\item The distance between the endpoints $x,y$ is 
	\begin{equation} \label{lem1_prop11}
		|x-y|<\delta/(10 c_\phi).
	\end{equation} 
\end{itemize}
Then 
\begin{equation}\label{lem1_concl1}
	|\gamma_{x,y}|\leq a |x-y|	
\end{equation}
and 
\begin{equation}\label{lem1_concl2}
	\min\{|\gamma_{x,z}|,|\gamma_{y,z}| \} \leq b \dOmega(z)
\end{equation} 
for all $z\in \gamma_{x,y}$, with $a= 8d_\phi$ and $b = 2d_\phi^2$. 
\end{theorem}
\begin{remark}
In the special case $\phi(t)= \log(1+t)$, conditions \eqref{tempC2_0} and \eqref{tempC2_00} are fulfilled by $c_\phi= 6 c^2$ and $d_\phi = 48c^2 e^{16 c^2}$.
\end{remark}
\begin{proof}
Without loss of generality we may assume that $\dOmega(x)\leq \dOmega(y)$. We consider two cases: when $\dOmega(y)>2|x-y|$ and when $\dOmega(y)\leq 2|x-y|$.

\textit{Case 1: $\dOmega(y)>2|x-y|$}

In this case for any $z$ on the straight line segment $[x,y]$, we have $\dOmega(z)\geq \frac{1}{2} \dOmega(y)$. Therefore
\begin{equation} 
\label{tempC1}
	\kOmega(x,y)\leq \int_{[x,y]} \frac{ds}{\dOmega(z)} \leq \frac{2|x-y|}{\dOmega(y)} \le 1.
\end{equation}
From this, \eqref{G-P1} and the fact that $\gamma_{x,y}$ is a quasihyperbolic geodesic, for any $z\in \gamma_{x,y}$,
\[
e^{-1} \le e^{-\kOmega(z,y)} \le \frac{\dOmega(z)}{\dOmega(y)} \leq e^{\kOmega(z,y)} 
\leq e.
\]
Hence, again applying \eqref{tempC1},
\[
|\gamma_{x,y}| \leq  e \dOmega(y) \int_{\gamma} \frac{ds}{\dOmega(z)} = e \dOmega(y) \kOmega(x,y) \leq 2e|x-y|
\]
and so by the lower bound on the ratio $\frac{\dOmega(z)}{\dOmega(y)}$,  
\[
|\gamma_{z,y}| \leq |\gamma_{x,y}| \leq e \dOmega(y) \leq e^2 \dOmega(z).
\]
So, for $\dOmega(y)>2|x-y|$, we established \eqref{lem1_concl1} and \eqref{lem1_concl2} with constants $a \geq  2e$ and $b\geq e^2$. 

\textit{Case 2: $\dOmega(y)\leq 2|x-y|$}

This case is more elaborate. We put $y_0:=y$ and define a sequence $\{{y}_i\}$ as follows: let $y_i$ be the point on $\gamma_{x,y}$ such that $\dOmega(y_i) = 2\dOmega(y_{i-1})$ and $|\gamma_{{y}_i,y}|$ is minimal. We fix $M$ to be the largest index such that  $\dOmega({y}_M)\leq 2|x-y|$. Next, we put $z_0:=x$ and $\{z_i\}$ to be the sequence defined inductively: $z_i$ is the point on $\gamma_{x,y}$ such that $\dOmega(z_i) = 2 \dOmega(z_{i-1})$ and $|\gamma_{x,z_i}|$ is minimal. We fix the largest index $N$ such that $\dOmega(z_N)\leq \dOmega(y)$. Then we split the points $z_i$ into two groups: $\{{z}_i\}_{i=0}^N$ and $\{{x}_i := z_{i+N}\}_{i=0}^M$. In particular, we have
\begin{equation} \label{temp_tilde_seq}
	\dOmega(z_j) \leq \dOmega(y)\leq 2|x-y|, \qquad j=0,\dots, N,
\end{equation}
and 
\begin{equation} \label{temp_bar_seq}
	\dOmega({y}_{i-1})<\dOmega({x}_i) \leq \dOmega({y}_i)\leq 2|x-y|, \qquad i=1,\dots M.
\end{equation} 
Note that by the choice of $z_i,x_j,y_j$, the subarcs $\gamma_{z_{i-1},z_{i}}$, $i=1, \dots , N,$  $\gamma_{x_{j-1},x_{j}}$, $\gamma_{y_{j-1},y_{j}}$, $j=1, \dots , M$  and $\gamma_{x_M,y_M}$ partition the curve $\gamma_{{x},{y}}$.

Our goal is to show the following control over these sub-arcs:
\begin{equation} \label{thm7_main_est_1a}
	|\gamma_{z_{i-1},z_i}| \leq c_\phi\dOmega(z_{i-1}),
\end{equation}
\begin{equation} \label{thm7_main_est_1b}
	\dOmega(z) \geq d_\phi^{-1}\cdot \dOmega(z_{i-1}) \quad \text{ for all } z\in \gamma_{z_{i-1},z_i},
\end{equation}
\begin{equation} \label{thm7_main_est_2a}
	\begin{cases}
		|\gamma_{x_{i-1},x_i}| \leq c_\phi \dOmega(x_{i-1}), \\
		|\gamma_{y_{i-1},y_i}| \leq c_\phi \dOmega(y_{i-1}), 	
	\end{cases}
\end{equation}
\begin{equation} \label{thm7_main_est_2b}
	\begin{cases}
		\dOmega(z) \geq d_\phi^{-1}\cdot \dOmega(x_{i-1}) \quad \text{ for all } z\in \gamma_{x_{i-1},x_i}, \\
		\dOmega(z) \geq d_\phi^{-1}\cdot \dOmega(y_{i-1}) \quad \text{ for all } z\in \gamma_{y_{i-1},y_i},	
	\end{cases}
\end{equation}
\begin{equation} \label{thm7_main_est_3a}
	|\gamma_{x_{M},y_M}| \leq  d_\phi \dOmega(x_M),
\end{equation}
and
\begin{equation} \label{thm7_main_est_3b}
	\dOmega(z) \geq d_\phi^{-1}\cdot \dOmega(x_M), \quad \text{ for all } z\in \gamma_{x_M,y_M}. 
\end{equation}

The theorem will be proved once we establish \eqref{thm7_main_est_1a} - \eqref{thm7_main_est_3b}. Indeed by \eqref{thm7_main_est_1a}, \eqref{thm7_main_est_2a}, \eqref{thm7_main_est_3a} followed by \eqref{temp_bar_seq} and \eqref{temp_tilde_seq} we get 
\[
|\gamma_{x,y}|\leq \sum_{i=1}^N |\gamma_{z_{i-1},z_i}| + \sum_{i=1}^M |\gamma_{x_{i-1},x_i}| + \sum_{i=1}^M |\gamma_{y_{i-1},y_i}| + |\gamma_{x_{M},y_M}| \leq 
\]
\[
\leq c_\phi \left[ \sum_{i=1}^N  \dOmega(z_{i-1}) + \sum_{i=1}^M  \dOmega(x_{i-1}) + \sum_{i=1}^M  \dOmega(y_{i-1})\right] + d_\phi \dOmega(x_M) \leq 
\]
\[
\leq c_\phi \left[  \dOmega(z_{N}) + \dOmega(x_{M}) + \dOmega(y_{M})\right] + d_\phi \dOmega(x_M) \leq 
\]
\[
\leq 6c_\phi |x-y| + 2d_\phi 
|x-y| \leq 8d_\phi
|x-y|,
\]
which implies \eqref{lem1_concl1}. Moreover, for any $z\in \gamma_{z_{i-1},z_i}\subset \gamma_{x,x_0}$, by \eqref{thm7_main_est_1a} and \eqref{thm7_main_est_1b} we get 
\[
|\gamma_{z_0,z}| \leq |\gamma_{z_0,z_i}| = \sum_{j=1}^{i} |\gamma_{z_{j-1},z_j}| \leq 2 c_\phi \dOmega(z_{i-1}) \leq (2c_\phi d_\phi) 
\dOmega(z).
\] 
This immediately implies \eqref{lem1_concl2} for all $z\in \gamma_{x,z_N}$ with $b\geq 2c_\phi d_\phi$. In the same way, \eqref{thm7_main_est_2a} and \eqref{thm7_main_est_2b} imply \eqref{lem1_concl2} for $z\in \gamma_{x_0,x_M} \cup \gamma_{y_0,y_M}$, and \eqref{thm7_main_est_3a} and \eqref{thm7_main_est_3b} yield \eqref{lem1_concl2} for $z\in \gamma_{x_M,y_M}$ with $b\geq 2(d_\phi)^2$.

Therefore, it only remains to prove \eqref{thm7_main_est_1a} - \eqref{thm7_main_est_3b}.

\underline{Estimates \eqref{thm7_main_est_1a} and \eqref{thm7_main_est_1b}:}
We prove the estimates by induction. To verify them for $i=1$, we need to note two inequalities. The first one is  
\begin{equation} \label{tempC2_S1_1}
	\frac{|\gamma_{z_{i-1},z_{i}}|}{\dOmega(z_{i-1})} \leq 2\kOmega(z_{i-1},z), 
\end{equation}
which holds for $i=1,2,\dots, N$  because our choice of $z_i$ guarantees that $\dOmega(z)\geq 2 \dOmega(z_{i-1})$ for any $z\in \gamma_{z_{i-1},z_{i}}$. The second inequality is 
\begin{equation} \label{tempCC2}
	\kOmega(z_{0},z_{1}) \leq c \left[1+\phi\left(\frac{|z_{0}-z_{1}|}{\dOmega(z_{0})}\right)\right]. 
\end{equation}
It is true because either $|z_0-z_1|< |x-y|<\delta$ and its validity is ensured by the hypothesis \eqref{lem1_prop1}, or $|z_0-z_1|\geq |x-y|$ and by the monotonicity  of $\phi$
\[
\kOmega(z_{0},z_{1}) \leq \kOmega(x,y) \leq c \left[1+\phi\left(\frac{|x-y|}{\dOmega(z_{0})}\right)\right] \leq c \left[1+\phi\left(\frac{|z_0-z_1|}{\dOmega(z_{0})}\right)\right].
\]
Combining \eqref{tempC2_S1_1}, \eqref{tempCC2} and \eqref{tempC2_0} we get 
\[
\frac{|\gamma_{z_{0},z_{1}}|}{\dOmega(z_{0})} \leq c_\phi,
\]
which, after being plugged back into \eqref{tempCC2}, yields 
\[
\kOmega(z_{0},z_1)\leq c \left[1+\phi\left(c_\phi\right)\right] \leq c_\phi,
\]
where the second inequality follows from \eqref{tempC2_0}. 
Combining these estimates with \eqref{G-P1} we get, for $z\in \gamma_{z_{0},z_1}$,
\[
|\gamma_{z_{0},z}| \leq |\gamma_{z_{0},z_{1}}|\leq c_\phi \dOmega(z_{0})\leq c_\phi e^{c_\phi}\dOmega(z) \leq d_\phi \; \dOmega(z),
\]
where the last inequality follows from the positivity of $\phi$ and hypothesis \eqref{tempC2_00}. 
This shows both \eqref{thm7_main_est_1a} and \eqref{thm7_main_est_1b} for $i=1$. 

Assume now that \eqref{thm7_main_est_1a} and \eqref{thm7_main_est_1b} hold for $i=1,2,\dots, j-1$. In order to see that it must also hold for $i=j$, we can repeat the argument employed in the base case $i=0$ provided that we have 
\begin{equation} \label{tempCC22}
	\kOmega(z_{j-1},z_{j}) \leq c \left[1+\phi\left(\frac{|z_{j-1}-z_{j}|}{\dOmega(z_{j-1})}\right)\right]. 
\end{equation}
To show the last inequality, we invoke the induction hypothesis to get 
\[
|z_{j-1} - y| \leq |x-y| + \sum_{i=1}^{j-1} |\gamma_{z_{i-1},z_i}| \leq |x-y| + c_\phi \dOmega(z_{j-1}) 
\] 
\begin{equation}\label{tempC2_S3}
	\leq |x-y| + (c_\phi/2)\; \dOmega(y) \leq (1+c_\phi) |x-y|<\delta.
\end{equation}
Then either $|z_{j-1} - z_{j}|< |z_{j-1} - y|$, and by \eqref{tempC2_S3} we can use hypothesis \eqref{lem1_prop1} to get \eqref{tempCC22} for $i=j$, or $|z_{j-1} - z_{j}|\geq |z_{j-1} - y|$ and again due to \eqref{tempC2_S3}, we get
\[
\kOmega(z_{j-1},z_{j}) \leq \kOmega(z_{j-1},y) \leq c \left[1+\phi\left(\frac{|z_{j-1}-z_{j}|}{\dOmega(z_{j-1})}\right)\right].
\]

\underline{Estimates \eqref{thm7_main_est_2a} and \eqref{thm7_main_est_2b}:}
As the proof of these inequalities is quite similar to what we have just done above, we only outline it. Start by noting that 
\begin{equation} \label{tempC2_S2_1}
	\frac{|\gamma_{{x}_{i-1},{x}_i}|}{\dOmega({x}_{i-1})} \leq 2\kOmega({x}_{i-1},{x}_i), \qquad 	\frac{|\gamma_{{y}_{i-1},{y}_i}|}{\dOmega({y}_{i-1})} \leq 2\kOmega({y}_{i-1},{y}_i)
\end{equation}
and that estimate \eqref{thm7_main_est_1a} provides us with 
\begin{equation} \label{tempC2_S2_temp1}
	|{x}_0- y| \leq |x-y| + |x-{x}_0| \leq |x-y| + |\gamma_{z_0,z_N}|\leq (1+2c_\phi)|x-y|<\delta.
\end{equation}

Then considering the cases $|{x}_0 - {x}_1|< |{x}_0 - y|$, $|{x}_0 - {x}_1|\geq |{x}_0 - y|$, $|y_0-y_1|<|y-x_1|$ and $|y_0-y_1|\geq |y-x_1|$ we can show that inequalities 
\begin{equation}\label{tempC2_S2_2a}
	\kOmega({x}_{i-1},{x}_{i}) \leq c \left[1+\phi\left(\frac{|{x}_{i-1}-{x}_{i}|}{\dOmega({x}_{i-1})}\right)\right]
\end{equation}
and 
\begin{equation}\label{tempC2_S2_2b}
	\kOmega({y}_{i-1},{y}_{i}) \leq c \left[1+\phi\left(\frac{|{y}_{i-1}-{y}_{i}|}{\dOmega({y}_{i-1})}\right)\right]
\end{equation}
hold for $i=1$. Next, combine \eqref{tempC2_S2_1} with \eqref{tempC2_S2_2a}, \eqref{tempC2_S2_2b} and \eqref{tempC2_0} to get 
\[
\frac{|\gamma_{{x}_{0},{x}_1}|}{\dOmega({x}_{0})}, \frac{|\gamma_{{y}_{0},{y}_1}|}{\dOmega({y}_{0})}  \leq c_\phi.
\]
Plugging back into \eqref{tempC2_S2_2a} and \eqref{tempC2_S2_2b}, we get 
\[
\kOmega({x}_{0},{x}_{1}), \kOmega({y}_{0},{y}_{1}) \leq c_\phi. 
\]
This gives \eqref{thm7_main_est_2a} and \eqref{thm7_main_est_2b} for $i=1$ as the last two inequalities in combination with \eqref{G-P1} give 
\[
|\gamma_{x_{0},z}| \leq |\gamma_{x_{0},x_{1}}|\leq c_\phi \dOmega(x_{0})\leq c_\phi e^{c_\phi}\dOmega(z) \leq d_\phi \; \dOmega(z)
\]
if $z\in \gamma_{x_0,x_1}$, and 
\[
|\gamma_{y_{0},z}| \leq |\gamma_{y_{0},y_{1}}|\leq c_\phi \dOmega(x_{0})\leq c_\phi e^{c_\phi}\dOmega(z) \leq d_\phi \; \dOmega(z)
\]
if $z\in \gamma_{y_0,y_1}$.

Assuming now that \eqref{thm7_main_est_2a} and \eqref{thm7_main_est_2b} hold for $i=1,2,\dots,j-1$ we can deduce 
\[
|x_{j-1} - y_{j-1}| \leq |x-y| + |\gamma_{z_0,z_N}| + \sum_{i=1}^{j-1} |\gamma_{x_{i-1},x_i}| + \sum_{i=1}^{j-1} |\gamma_{y_{i-1},y_i}| < \delta. 
\]

This inequality allows us to consider two cases $|{x}_{j-1} - {x}_{j}|< |{x}_{j-1} - {y}_{j-1}|$ and $|{x}_{j-1} - {x}_{j}|\geq |{x}_{j-1} - {y}_{j-1}|$ and get \eqref{tempC2_S2_2a} for $i=j$. Similarly, the inequality 
\[
|{y}_{j-1} - {x}_{j}| \leq |x-y| + |\gamma_{z_0,z_N}| +  \sum_{i=1}^{j-1} |\gamma_{{x}_{i-1},{x}_i}|+\sum_{i=1}^{j-1} |\gamma_{{y}_{i-1},{y}_i}| + |\gamma_{{x}_{j-1},{x}_j}| < \delta,
\] 
allows us to consider two cases: $|{y}_{j-1} - {y}_{j}|< |{y}_{j-1} - {x}_{j}|$ and $|{y}_{j-1} - {y}_{j}|\geq |{y}_{j-1} - {x}_{j}|$, and get \eqref{tempC2_S2_2b} for $i=j$.

\underline{Estimates \eqref{thm7_main_est_3a} and \eqref{thm7_main_est_3b}:}
The proof of these estimates follows a slightly different line of reasoning. Unlike in the proofs of \eqref{thm7_main_est_1a} - \eqref{thm7_main_est_2b}, estimate 
\begin{equation}\label{tempC2_S3_0}
	\kOmega({x}_M,{y}_M) \leq c\left[1+ \phi\left(\frac{|{x}_M-{y}_M|}{\dOmega({x}_M)}\right)\right]
\end{equation}
is readily available by hypothesis \eqref{lem1_prop1} thanks to the estimates we have already obtained:
\begin{equation} \label{tempC2_S3_00}
	|{x}_M - {y}_M |\leq |x-y| + |\gamma_{x,z_N}| +  |\gamma_{x_0,x_M}| + |\gamma_{y,y_M}| < \delta.
\end{equation}
What needs to be shown, however, is the following analogue of inequalities \eqref{tempC2_S1_1} and \eqref{tempC2_S2_1}:
\begin{equation} \label{tempC2_S3_1}
	\frac{|\gamma_{{x}_M,{y}_M}|}{\dOmega({x}_M)} \leq K \kOmega({x}_M,{y}_M)
\end{equation}
for some $K\leq e^{8c_\phi}$.

First of all, note that if $\dOmega(z)<2 \dOmega({x}_M)$ for all $z\in \gamma_{{x}_M,{y}_M}$, then \eqref{tempC2_S3_1} holds with $K=2$. Otherwise, there is $z\in \gamma_{{x}_M,{y}_M}$ with $\dOmega(z)\geq 2 \dOmega({x}_M)$. By our choice of the points ${x}_M$ and ${y}_M$ this means that 
\begin{equation} \label{tempC2_S3_2}
	\dOmega({x}_M) >|x-y|/2. 
\end{equation}

Combining \eqref{tempC2_S3_00}, \eqref{tempC2_S3_0} and \eqref{tempC2_S3_2} we get
\[
\kOmega({x}_M,{y}_M) \leq c\left[1+ \phi\left(\frac{|x_M-y_M|}{\dOmega(x_M)}\right)
\right]\leq c\left[1+ \phi\left(2+6c_\phi\right)\right] \leq 2+6c_\phi.
\]
By \eqref{G-P1} this means that for any $z\in \gamma_{{x}_M,{y}_M}$, we have
\begin{equation} \label{tempC2_S3_3}
	\left| \log\left(\frac{\dOmega(z)}{\dOmega({x}_M)} \right) \right|\leq \kOmega({x}_M,{y}_M) \leq  2+6c_\phi \leq 8 c_\phi,
\end{equation}
which shows \eqref{thm7_main_est_3b}. Moreover, the last estimate yields \eqref{tempC2_S3_1} with $K=e^{8c_\phi}$:
\[
\frac{|\gamma_{{x}_M,{y}_M}|}{\dOmega({x}_M)} \leq \sup_{z\in \gamma_{{x}_M,{y}_M}} \frac{\dOmega(z)}{\dOmega({x}_M)} \cdot \int_{\gamma_{{x}_M,{y}_M}} \frac{ds(z)}{\dOmega(z)} \leq e^{8c_\phi} \kOmega(\gamma_{{x}_M,{y}_M}). 
\] 

Finally, \eqref{tempC2_S3_0} and \eqref{tempC2_S3_1} yield
\[
\frac{|\gamma_{{x}_M,{y}_M}|}{\dOmega({x}_M)} \leq c e^{8c_\phi} \left[1+ \log\left(1+\frac{|\gamma_{{x}_M,{y}_M}|}{\dOmega({x}_M)}\right)\right],
\]
which by implication \eqref{tempC2_0} results in 
\begin{equation} \label{tempC2_S3_last}
	|\gamma_{{x}_M,{y}_M}| \leq d_\phi \dOmega({x}_M) \leq 2 d_\phi |x-y|.
\end{equation}

\end{proof}

\subsubsection{Distance version of Theorem \ref{G-O_TH1}}
The strengthening of Theorem \ref{G-O_TH1} concerns the John condition \eqref{defin_ed_GO_b}. Specifically, we want to show that the same conclusion holds if we weaken hypothesis  \eqref{defin_ed_GO_b} by replacing arclength by distance.

\begin{theorem} \label{thm_GO-1_our version}
Let $x,y\in \Omega\subset \mathbb{R}^n$ and $a,b\geq 1$. 
If there exists $\gamma_{x,y}$, a curve in $\Omega$ connecting $x$ and $y$ satisfying 
\begin{equation} \label{def_unif_domain_w_dist}
	\begin{cases}
		|\gamma_{x,y}| \leq a |x-y| \\
		\min \{|x-z|, |y-z|\} \leq b \dOmega(z), \quad \forall z\in \gamma_{x,y},
	\end{cases}
\end{equation}
then 
\[
\kOmega(x,y) \leq C_{a,b,n} \left[1+\jOmega(x,y)\right]
\]
with 
\[
C_{a,b,n} \leq C_{n} a^n b^n,
\]
where $C_{n}$ is some constant that depends on $n$ only. 

\end{theorem}
\begin{proof}
Note that by the quasiconvexity hypothesis  \eqref{defin_ed_GO_a} on $\gamma = \gamma_{x,y}$,
\[
\gamma \subset B(x;a|x-y|)\cap B(y,a|x-y|)=: O_{x,y}.
\]
We claim that there exists $\mathcal{A} = \{B_i\}$, a collection of closed balls $B_i$, such that $\mathcal{A}$ is a cover of $O_{x,y}$,  
\begin{equation} \label{temp_cover_cond1}
	\diam_{\kOmega}(B_i) \leq 2 \text{ when } B_i\cap\gamma\neq \varnothing,
\end{equation}
and 
\begin{equation} \label{temp_cover_cond2}			
	\# \mathcal{A} \leq C_{a,b,n} [1+\jOmega(x,y)].
\end{equation}
Assuming this claim, the proof is not difficult. Indeed, in this case we can find distinct $Q_0,Q_1,\dots, $ $Q_{N-1} \in \mathcal{A}$ such that $Q_0\ni x$, $Q_{N-1}\ni y$ and for $i\in [0,N-2]$ we have 
\[
Q_i\cap Q_{i+1}\cap \gamma \neq \varnothing.
\]
Then because $Q_i$ form a chain and $\ko$ is a metric, 
\[
\kOmega(x,y)  = \diam_{\kOmega}\{x,y\} \leq \diam_{\kOmega}\Big(\bigcup_{i=0}^{N-1} Q_i \Big) \leq N \max_{0\leq i \leq N-1} {\diam_{\kOmega}(Q_i)}.   
\]
Using condition \eqref{temp_cover_cond1} and \eqref{temp_cover_cond2} we deduce 
\[
\kOmega(x,y) \leq  2C_{a,b,n} [1+\jOmega(x,y)].
\]

So we just need to establish the claim. Our collection $\mathcal{A}$ will contain the closed balls 
$B(x,\frac{\dOmega(x)}{2})$  and $B(x,\frac{\dOmega(y)}{2})$.
Note that for any $z\in B(x,\frac{\dOmega(x)}{2})$, we have $\dOmega(z)\geq \frac{\dOmega(x)}{2}$ and therefore, as in \eqref{tempC1},
\[
\kOmega(x,z) \leq \int_{[x,z]} \frac{ds}{\dOmega(v)} 
\leq 1.
\]
This shows that $B(x,\frac{\dOmega(x)}{2})$ satisfies \eqref{temp_cover_cond1}. By the same argument, $B(y,\frac{\dOmega(y)}{2})$ satisfies \eqref{temp_cover_cond1}.

To specify the remaining balls, we fix $\theta\in (0,\frac1{12})$ and construct covers of the following regions: $B(x,\frac{|x-y|}{2})\setminus B(x,\frac{\dOmega(x)}{2})$, $B(y,\frac{|x-y|}{2})\setminus B(y,\frac{\dOmega(y)}{2})$ and  $O_{x,y}\setminus B(x,\frac{|x-y|}{2}) \setminus  B(y,\frac{|x-y|}{2})$. 

\medskip
\underline{$\mathcal{R}$, a cover of $O_{x,y}\setminus B(x,\frac{|x-y|}{2}) \setminus  B(y,\frac{|x-y|}{2})$:}

By the doubling properties of $\Rn$, any ball of radius $r$ can be covered by at most $C_n \theta^{-n}$ closed balls of radii $\theta r$, where $C_n$ depends on $n$ only. Therefore we may claim that there is a collection $\mathcal{R}$ of at most $C_n a^n b^n\theta^{-n}$ balls of radii $\frac{\theta}{b} \frac{|x-y|}{2}$ covering $O_{x,y}\setminus B(x,\frac{|x-y|}{2}) \setminus  B(y,\frac{|x-y|}{2})$. We will always assume that no ball in $\mathcal{R}$ is a subset of $B(x,\frac{|x-y|}{2})$ nor of $B(y,\frac{|x-y|}{2})$. Hence, for any ball $B\in \mathcal{R}$ of radius $r$ and any point $z$ in it, we have 
\[
|z-x|\geq \frac{|x-y|}{2} (1-2\theta/b), \quad \text{ and } \quad  |z-y|\geq \frac{|x-y|}{2} (1-2\theta/b)
\]
or 
\begin{equation} \label{temp_decomp_3}
	\min\left(\frac{|x-z|}{r}, \frac{|y-z|}{r} \right) \geq \frac{b}{\theta}  - 2 > 10 b. 
\end{equation}
In particular, if $B$ intersects $\gamma$ we can choose $z$ to be a point in this intersection. Then for any $w\in B$,
\[
\dOmega(w) \geq \dOmega(z) - |z-w| \geq \frac{1}{b}\min(|x-z|,|y-z|) - 2r  \geq 8 r,
\]
where we use the distance John condition in the second inequality and \eqref{temp_decomp_3} for the last one. 
This means that given $B\in \mathcal{R}$ with $z\in B\cap \gamma$, we have 
\[
\kOmega(z,w) \leq \int_{[z,w]} \frac{ds}{\dOmega(v)} \leq \frac{2r}{8r} \leq \frac{1}{4} \quad \text{ for all } w\in B,
\]
or 
\[
\diam_{\kOmega}(B) \leq \frac{1}{2}. 
\]

\medskip
\underline{$\mathcal{E}$, a cover of $B(x,\frac{|x-y|}{2})\setminus B(x,\frac{\dOmega(x)}{2})$:}

The above-mentioned doubling property of $\Rn$ allows us to claim that there is a collection $\mathcal{E}$ of no more than $C_n b^n \theta^{-n} (2+\log\frac{|x-y|}{\dOmega(x)})$ closed balls covering $B(x,\frac{|x-y|}{2})\setminus B(x,\frac{\dOmega(x)}{2})$ with the following property: if $B\in \mathcal{E}$ is a ball of radius $r$, then for any $z\in B$
\begin{equation} \label{temp_decomp}
	\frac{|z-x|}{r} \geq  4 b.
\end{equation}
To see this, we write 
\[
B\Big(x,\frac{|x-y|}{2}\Big)\setminus B\Big(x,\frac{\dOmega(x)}{2}\Big) \subset  \bigcup_{j=1}^{N} B\Big(x,\frac{|x-y|}{2^j}\Big)\setminus B\Big(x,\frac{|x-y|}{2^{j+1}}\Big), 
\]
where $N$ is the smallest integer satisfying $2^N \geq \frac{|x-y|}{\dOmega(x)}$. Then we note that each annulus $ B(x,\frac{|x-y|}{2^j})\setminus B(x,\frac{|x-y|}{2^{j+1}})$ can be covered by at most $C_n b^n \theta^{-n}$ balls of radii $\frac{\theta |x-y|}{2^{j}b}$. Moreover, for any such ball $B(u,\frac{\theta |x-y|}{2^{j}b})$ and for any $z$ in it, we have 
\[
|z-x| \geq \frac{|x-y|}{2^{j+1}} - \frac{2\theta |x-y|}{2^{j}b} = \frac{|x-y|}{2^j} \left[\frac{1}{2} - \frac{2 \theta}{b}\right] = \left[\frac{b}{2\theta} - 2\right] r \geq 4b r.
\]
In particular, if $B\in \mathcal{E}$ intersects $\gamma$ at $z$ and $w\in B$ is arbitrary, then 
\[
\dOmega(w) \geq \dOmega(z) - 2 r \geq r\left(\frac{|x-z|}{b} - 2   \right) \geq 2r. 
\]
Hence 
\[
\kOmega(w,z)\leq \int_{[z,w]} \frac{ds}{\dOmega(v)}\leq 1. 
\]
Summing up, we conclude that for any $B\in \mathcal{E}$ intersecting $\gamma$ we have 
\[
\diam_{\kOmega}(B)\leq 2. 
\]

\medskip
\underline{$\mathcal{F}$, a cover of $B(y,\frac{|x-y|}{2})\setminus B(y,\frac{\dOmega(y)}{2})$}

By the same token, we may claim that there is a collection $\mathcal{F}$ of no more than $C_n b^n \theta^{-n} (2+\log\frac{|x-y|}{\dOmega(y)})$ closed balls covering $B(y,\frac{|x-y|}{2})\setminus B(y,\frac{\dOmega(y)}{2})$ with the following property: if $B\in \mathcal{F}$ intersects $\gamma$ 
then 
\[
\diam_{\kOmega}(B) \leq 2.
\]

Finally, as we verified all the balls in the collections $\mathcal{E}, \mathcal{F}$ and $\mathcal{R}$ satisfy \eqref{temp_cover_cond1} and there are no more than $C_n b^n \theta^{-n} (2+\log\frac{|x-y|}{\dOmega(x)})$, $C_n b^n \theta^{-n} (2+\log\frac{|x-y|}{\dOmega(y)})$ and  $C_n a^n b^n \theta^{-n}$ of them in each cover respectively, if we put $A$ to be the union of $\mathcal{E}, \mathcal{F}$ and $\mathcal{R}$, together with $B(x,\frac{\dOmega(x)}{2})$  and $B(x,\frac{\dOmega(y)}{2})$, then we get a collection of closed balls satisfying both  \eqref{temp_cover_cond1} and \eqref{temp_cover_cond2}. 		

\end{proof}

\subsection{Locally uniform domains}
\label{sec-locallyuniform}
One can give various definitions of locally uniform domains. We consider some of them below and use Theorem \ref{thm_GO-2_our version} and  \ref{thm_GO-1_our version} to show their equivalence. 

We start with the most natural localization of the Martio-Sarvas definition with conditions \eqref{defin_ed_GO_a} and \eqref{defin_ed_GO_b}. 
\begin{definition}[Herron-Koskela]
A domain $\Omega$ is a \textit{locally} (length)-uniform domain if there exists $a,b\geq 1$ and $\delta>0$ such that for all $x,y\in \Omega$ with $|x-y|<\delta$, there is $\gamma(x,y)\subset \Omega$ such that \eqref{defin_ed_GO_a} and \eqref{defin_ed_GO_b} hold.
\end{definition}

As a useful preliminary  definition, we also consider 
\begin{definition}
A domain $\Omega$ is a \textit{locally} \textbf{distance} uniform domain if there exists $a,b\geq 1$ and $\delta>0$ such that for all $x,y\in \Omega$ with $|x-y|<\delta$, there is $\gamma(x,y)\subset \Omega$ such that \eqref{def_unif_domain_w_dist} holds.
\end{definition}

\begin{definition}[Jones \cite{jones2}]
A domain $\Omega$ is called an $(\epsilon,\delta)$ domain if there exists $\epsilon>0$ such that for all $x,y\in \Omega$ with $|x-y|< \delta$
there is a rectifiable curve $\gamma\subset \Omega$ joining $x$ to $y$ such that 
\begin{equation}
	\label{q-convex_cond}
	s(\gamma) \leq \epsilon^{-1}|x-y|
\end{equation}
and 
\begin{equation}
	\label{John_cond}
	\dom(z) \geq \epsilon\frac{|z-x||z-y|}{|x-y|} \quad \forall z\in \gamma(x,y).
\end{equation} 
\end{definition} 
\begin{theorem} \label{thm_all_domains_equiv}
Let $\Omega\subset \Rn$ be a domain. Then the following are equivalent:
\begin{itemize}
	\item $\Omega$ is an $(\epsilon,\delta)$-domain for some $\epsilon,\delta>0$;
	\item $\Omega$ is a locally distance uniform domain;
	\item $\Omega$ is a locally uniform domain.
	
\end{itemize} 
\end{theorem} 
\begin{proof}
It is not difficult to see that the Jones' definition is equivalent to the definition of locally distance uniform domains. More precisely (see e.g.\ \cite{vaisala}), if $\Omega$ is an $(\epsilon,\delta)$ domain then $\Omega$ is locally distance uniform with the same $\delta$ and $a = 1/\epsilon$, $b=2/\epsilon$. Conversely, if $\Omega$ is locally distance uniform, then $\Omega$ is an $(\epsilon,\delta)$-domain with the same $\delta$ and $\epsilon= 1/(ab)$. 

Suppose that $\Omega$ is locally distance uniform. Applying Theorem \ref{thm_GO-1_our version} and Theorem \ref{thm_GO-2_our version} we see that $\Omega$ is locally (length) uniform with a possibly smaller $\delta'$. Finally, length uniform domains are distance uniform because arclength dominates distance.
\end{proof}

\section{bmo extension domains are locally uniform}
\label{sec-necessity}
In this section we will prove the necessity in Theorem~\ref{thm_ed_is_extension}.  We show it under a weaker hypothesis, in the form of the following theorem. 
\begin{theorem}\label{thm_bmo_ext_is_uniform}
Let  $\Omega$ be a domain and suppose that $\lambda > 0$, $C > 0$ are such that any $f\in \bmolo$ is the restriction to $\Omega$ of some $F\in \BMO(\Rn)$ with 
\[
\|F\|_{\BMO(\Rn)} \leq C \|f\|_{\bmo_\lambda(\Omega)}.
\]
Then $\Omega$ is a locally uniform domain. 
\end{theorem}

\subsection{Three preliminary lemmas}
This subsection is devoted to establishing three technical lemmas that will be used in the next subsection. 
We start with the following fact that is valid in the more general setting of doubling metric measure spaces (see Lemma 18 in \cite{VG}). 
\begin{lem}
\label{lem_k_is_BMO}
For any fixed $a\in \Omega$, the function $f(x) = \ko(a,x)$ is an element of $\BMOo$ and 
\[
\|f\|_{\BMOo} \leq c,
\]
where $c$ depends only on the dimension $n$. 
\end{lem}
\begin{proof}

We will use the following result, sometimes known as a {\em local-to-global} property, due to Reimann and Rychener \cite{reimann_rychener} (also attributed to Jones  - see  Theorem A1.1 and Corollary A1.1 in \cite{BrezisNirenberg}): there is a constant $C$ that depends only on the dimension $n$ such that for $f \in \Loneloc(\Omega)$, 
\begin{equation}
	\label{loc-glob}
	\|f\|_{\BMOo} \leq C \sup\limits_{2Q \subset \Omega} \fint_{Q} |f(x)  - f_{Q}| dx.
\end{equation}
Here the supremum is taken over all cubes $Q$ whose doubles $2Q$ are contained in $\Omega$ (the notation $cQ$ denotes the  concentric cube of $c$ times the sidelength).
Since
$$
\fint_{Q} |f(x)  - f_{Q}| dx \leq 
\frac{1}{|Q|^2} \int_{Q} \int_{Q} |f(x)  - f(y)| dx dy,
$$ 
it is enough to establish that 
$$
\sup\limits_{2Q \subset \Omega} \frac{1}{|Q|^2} \int_{Q} \int_{Q} |\ko(a,x)  - \ko(a,y)| dx dy \leq C'
$$
for some $C'$ depending at most on $n$. 

Suppose $x, y \in Q, 2Q \subset \Omega$.  Since $\ko$ is a distance, the triangle inequality gives
\begin{equation}
	\label{eq-osc_k}
	|\ko(a,x)  - \ko(a,y)| \leq \ko(x, y) \leq \int_{[x,y]} \frac{ds}{\dom(z)} \leq \frac{\diam(Q)}{\dist(Q,\bOmega)}, 
\end{equation}
where the integral on the right is over the line segment $[x,y]$. Since $2Q \subset \Omega$ means $\dist(Q,\bOmega) \geq \ell(Q)/2 = \diam(Q)/2\sqrt{n}$, 
we get $|\ko(a,x)  - \ko(a,y)| \leq  2\sqrt{n}$.
\end{proof}

\begin{remark}
\label{rem-osc_k}
Let $Q$ be a Whitney cube of $\Omega$ (see Section~\ref{sec_whitney_cubes}).  Estimate \eqref{eq-osc_k}, combined with property \eqref{WC2_dist_vs_size} below, gives
$$\sup_{x,y\in Q}|\ko(a,x)  - \ko(a,y)| \leq \sqrt{n}.$$
This is true (with a larger constant) if $x$ and $y$ are in adjacent Whitney cubes.  From this we can deduce Jones' observation, in \cite{jones1}, that the quasihyperbolic distance $\ko(x,y)$ is equivalent to the length of a shortest Whitney chain between the Whitney cubes containing $x$ and $y$.
\end{remark}

The following lemma provides us with control of the $\bmo$ norm by the $\BMO$ norm of the function and its $L^\infty$ norm {\em away from the boundary}.  To make this precise, we first define what  we will refer to as the {\em interior region} and the quasihyperbolic distance to this set.

\begin{definition}
Given a domain $\Omega$ and $\lambda>0$ we denote by $\Omegal$ the set $\{x\in \Omega: \dom(x)\geq \lambda/4\}$ and 
define, for $x \in \Omega$,
$$\ko(x, \Omegal) := \inf_{p \in \Omegal} \ko(x,p).$$
\end{definition}

\begin{lem}\label{lem_BMO_plus_LInf}
Let $f\in \BMOo$, $\lambda > 0$. Then 
\begin{equation}
	\label{temp_easy_1}
	\|f\|_\bmolo \lesssim  \sup_{2Q \subset \Omega} \fint_Q |f(x) - f_Q| dx  + \sup_{2Q \subset \Omega, \ell(Q) \geq \lambda/2} |f|_Q < \infty.
\end{equation} 
In particular, if $f\in L^\infty(\Omegal)$, then 
$$
\|f\|_{\bmolo} \lesssim \|f\|_{\BMOo} + \|f\|_{L^{\infty}(\Omegal)}.
$$
\end{lem}

Here and below, we use the notation $u \lesssim v$ when there exists a constant $c$ such that $u \leq cv$.  

\begin{proof}
The first term on the right-hand-side of \eqref{temp_easy_1} controls $ \|f\|_{\BMOo}$ by \eqref{loc-glob}.  To control the averages over large cubes, let
$Q \subset \Omega$ be any cube such that $\ell(Q)\geq \lambda$. We take $Q_0$ to be the cube co-centered with $Q$ with $\ell(Q_0) = \ell(Q)/2$. Then $2Q_0 \subset Q \subset \Omega$, 
and $$ 
|f|_Q \leq \fint_Q |f(x) - f_{Q_0}| + |f|_{Q_0} \leq \fint_Q \fint_{Q_0}  |f(x) - f(y)| dx dy + |f|_{Q_0} \leq 2^{n+1}  \fint_Q |f(x) - f_{Q}| + |f|_{Q_0}.$$
The integral in the first term on the right-hand-side is controlled by $\|f\|_{\BMOo}$, and therefore, applying \eqref{loc-glob} again, the right-hand-side is controlled by the right-hand-side of \eqref{temp_easy_1}.

Finally, note that if $2Q \subset \Omega$ and $\ell(Q) \geq \lambda/2$, then $\dist(Q, \bOmega) \geq \lambda/4$, i.e.\ $Q \subset \Omegal$.
\end{proof}

The third lemma says that we can always talk about quasihyperbolic geodesics from the interior region $\Omegal$ to the points in $\Omega\setminus \Omegal$.

\begin{lem} \label{lem_set_geodesic_exist}
Let $\lambda>0$. For each $x\in \Omega\setminus \Omegal$ there exists a point $x'\in \Omegal$ and a curve $\gamma(x,x')$ such that 
$$
\ko(x,\Omegal) = \ko(x,x')=  \int_{\gamma(x,x')} \frac{ds}{\dom(z)} .
$$
\end{lem}
\begin{proof}
Take  $x\in \Omega\setminus \Omegal$.  We will only show that $\ko(x,\Omegal) = \ko(x,x')$ for some $x'\in \Omegal$ as the existence of a geodesic $\gamma(x,x')$ is proved by Lemma 1 in \cite{gehring_osgood}. 

First of all, note that $\Omegal$ is a closed set and 
$
\ko(x,\cdot)
$
is a continuous function, so if $\Omegal$ is bounded then 
$\inf\limits_{y\in \Omegal} \ko(x,y)$
is attained in at some point in $\Omegal$.    

For unbounded $\Omegal$, fix a point $y_0$ in $\Omegal$ and set $R= \max(\lambda \, \ko(x,y_0), |x-y_0|)$. We claim that 
$$
\inf_{y\in \Omegal} \ko(x,y) = \inf_{y'\in \Omegal \cap B_R(x)} \ko(x,y'),
$$
where $B_R(x)$ is the closed ball centered at $x$ of radius $R$. As the latter infimum is attained at some point in $\Omegal \cap B_R(x)$, all we need to finish the proof is to show that
if $y\in \Omegal$ with $|x-y|>R$ then		
$$\ko(x,y) \geq  \inf_{y'\in \Omegal \cap B_R(x)} \ko(x,y').$$

Let  $\gamma(x,y)$ be a quasihyperbolic geodesic between $x$ and $y$. If there is a $z\in\gamma(x,y)$ such that $z \in \Omegal \cap B_R(x)$, then 
$$
\ko(x,y)  \geq \ko(x,z) \geq \inf_{y'\in \Omegal \cap B_R(x)} \ko(x,y'). 
$$
Otherwise, $\gamma(x,y)\cap B_R(x)\cap \Omegal = \varnothing$ and by our choice of $R$
\[
\ko(x,y)\geq \int_{\gamma \cap B_R(x)} \frac{ds}{\dom(z)} \geq \frac{s(\gamma \cap B_R(x))}{\lambda/4} \geq \frac{R}{\lambda} \geq \ko(x,y_0) \geq  \inf_{y'\in \Omegal \cap B_R(x)} \ko(x,y'). 
\] 
\end{proof}

\subsection{Proof of Theorem \ref{thm_bmo_ext_is_uniform}}
The proof is based on Lemmas \ref{lem_nec_main1} and \ref{lem_nec_main2} and Theorem \ref{thm_GO-2_our version}. Lemmas \ref{lem_nec_main1} and \ref{lem_nec_main2} below will show the existence of  $\delta_\lambda>0$ and $C'\geq 1$ such that inequality $\ko(x_1,x_2)\leq C' (1+\jo(x_1,x_2))$ holds for all $|x_1-x_2|\leq \delta_\lambda$. Thus, by Theorem \ref{thm_GO-2_our version} any quasihyperbolic geodesic $\gamma_{x,y}$ connecting points $x,y\in \Omega$ with $|x-y|<\delta_\lambda/(60C'^2)$ satisfies conditions \eqref{defin_ed_GO_a} and \eqref{defin_ed_GO_b} with some $a,b$ that depend on $C'$ only.  

\begin{lem} \label{lem_nec_main1}
Let  $\Omega$, $\lambda > 0$ and $C > 0$ be as in the hypothesis of Theorem~\ref{thm_bmo_ext_is_uniform}.
Then there exists $C'$ depending only on $C$ and $n$ such that for all $z_1,z_2\in \Omega$ and $R_1,R_2>0$ with 
\[
R_i \leq \min(\ko(z_i,\Omegal),\ko(z_1,z_2)), \quad i=1,2,
\] 
we have 
\[
R_1+R_2 \leq C'(\jo(z_1,z_2)+1). 
\]
\end{lem}
\begin{proof}
Put 
\[
f_1(x) = \max(R_1  - \ko(z_1,x), 0), \qquad f_2(x) = \max(R_2  - \ko(z_2,x), 0).
\]
Lemma \ref{lem_k_is_BMO} and the fact that truncation of a $\BMO$ function is still in $\BMO$ show that both $f_i\in \BMO(\Omega)$. Moreover, as $R_i\leq \ko(z_i,\Omegal)$, 
\[
f_i = 0 \mbox{ on } \Omegal, \ i=1,2.
\] 
Let $f=f_1 - f_2$.  Thus by Lemma \ref{lem_BMO_plus_LInf}, $f \in \bmolo$ with $\|f\|_\bmolo<c$ for some $c$ depending on $n$ only. 

Furthermore, $f_1(z_2) = f_2(z_1) = 0$ as $R_i\leq \ko(z_1,z_2)$ for both $i=1,2$. Thus $f(z_1) - f(z_2) = R_1+R_2$. Let $Q_i$ be the Whitney cubes containing $z_i$.  By Remark \ref{rem-osc_k}
\begin{equation} \label{eq-pts_to_cubes22}
	R_1 + R_2  \leq  |f_{Q_1} - f_{Q_2}| + \sup_{y\in Q_1} |f(z_1) - f(y)| + \sup_{y\in Q_2} |f(z_2) - f(y)| \leq |f_{Q_1} - f_{Q_2}| + 2 \sqrt{n}.	
\end{equation}

It remains to estimate $|f_{Q_1} - f_{Q_2}|$. As $f\in \bmolo$ and $\Omega$ is assumed to be an extension domain, we may apply Lemma 2.1 in \cite{jones1} to a $\BMO(\Rn)$ extension of $f$ and obtain 
\[
|f_{Q_1} - f_{Q_2}|\leq C'' d_2(Q_1,Q_2):= C'' \left(\left|\log\frac{l(Q_1)}{l(Q_2)}\right| + \log\left(2+ \frac{\textrm{dist}(Q_1,Q_2)}{l(Q_1)+l(Q_2)}\right) \right),
\] 
where $C''$ depends only on extension constant $C$ and $n$. 
By the properties of Whitney cubes, the right-hand-side is bounded by a constant times the following quantity  
$$\left|\log \frac{\dom(x_1)}{\dom(x_2)} \right| +  \log\left(2+\frac{|x_1 - x_2|}{\dom(x_1)+\dom(x_2)} \right).$$
Finally, since $\dom$ is Lipschitz with constant $1$, we can bound this quantity by a constant multiple of
$$\jo(x_1,x_2) =  \frac 1 2\log\left( 1+ \frac{|x_1-x_2|}{\dom(x_1)}  \right)\left( 1+ \frac{|x_1-x_2|}{\dom(x_2)} \right).$$
Combining the last three estimates with \eqref{eq-pts_to_cubes22} gives $\ko(x_1,x_2)  \leq C'\jo(x_1,x_2) + 2\sqrt{n}$.

\end{proof}
\begin{cor}
Under the assumptions of the preceding lemma, for any two points $u_1,u_2\in \Omega$ 
\begin{equation}\label{key_ext_estimate}
	\min\{\ko(u_1,u_2), \ko(u_1,\Omegal)+\ko(u_2,\Omegal) \} \leq C'(1+ \jo(u_1,u_2)),
\end{equation}
holds for some $C'$ that depends on the extension bound $C$ and dimension $n$ only. 
\end{cor}
\begin{proof}
	If  $\ko(u_1,u_2)\leq \ko(u_1,\Omegal)+\ko(u_2,\Omegal) $ we can assume that $\ko(u_1,\Omegal)+\ko(u_2,\Omegal)>0$, put 
	\[
	\theta:= \frac{\ko(u_2,\Omegal)}{\ko(u_1,\Omegal)+\ko( u_2,\Omegal) }, 
	\]
	and evoke the preceding lemma with  
	\[
	R_1 = (1-\theta)\ko(u_1,u_2), \qquad R_2= \theta\ko(u_1,u_2). 
	\]
	If $\ko(u_1,u_2)\geq \ko(u_1,\Omegal)+\ko(u_2,\Omegal) $, we evoke the lemma with $ R_i = \ko(u_i,\Omega) $. 
\end{proof}
\begin{lem} \label{lem_nec_main2}
Let domain $\Omega$, $\lambda > 0$ and $C'>0$ be as in Lemma~\ref{lem_nec_main1}. Then there is $\delta_\lambda>0$ such that for all $x_1,x_2\in \Omega$ with $|x_1-x_2|\leq \delta_\lambda$, 
\[
\ko(x_1,x_2)\leq C' (1+\jo(x_1,x_2)). 
\]  
\end{lem}
\begin{proof}
First, we note that if $\dom(x_1)\geq 2|x_1-x_2|$ or $\dom(x_2)\geq 2|x_1-x_2|$, then by the same argument as in \eqref{tempC1},
\[
\ko(x_1,x_2) 
\leq 1.
\]
Therefore we will focus on the case $\dom(x_1),\dom(x_2)< 2|x_1-x_2|$.

We will show that choosing $\delta_\lambda\in (0,\lambda/16)$ small enough so that 
\begin{equation} \label{how_small_delta_is}
	\log \frac{\lambda}{12\delta_\lambda} > C' + (C')^2(14+2\log C'),
\end{equation}
we get
\[
\ko(x_1,x_2)\leq C' (1+\jo(x_1,x_2))
\]
for any $x_1,x_2\in \Omega$ with $\dom(x_1),\dom(x_2)\leq 2|x_1-x_2|<2\delta_\lambda$.

Let $x_1,x_2\in \Omega$ be any such points.  By the assumption on $\delta_\lambda$,  $x_1,x_2\notin \Omegal$.  
Let $\gamma_1$ and $\gamma_2$ be quasihyperbolic geodesics from $\Omegal$ to $x_1$ and $x_2$, respectively, and $y_1,y_2\in \Omegal$ be the endpoints of these geodesics. Then  
\[
|x_i-y_i|\geq \dom(y_i) - \dom(x_i)=\lambda/4 - \dom(x_i)\geq \lambda/4 - 2\delta_\lambda > 2\delta_\lambda.
\]
Hence we can find points $z_1\in \gamma_1$ and $z_2\in \gamma_2$ such that $s(\gamma_i(x_i,z_i)) = \delta_\lambda$.

We will show that our choice of $\delta_\lambda$ provides us with inequality 
\begin{equation}\label{otc_z}
	\ko(z_1,z_2)<\ko(z_1,y_1)+\ko(z_2,y_2),
\end{equation}
which results in 
\[
\ko(x_1,x_2) \leq \ko(x_1,z_1) + \ko(z_1,z_2) + \ko(x_2,z_2)< 
\]
\[
\ko(x_1,z_1) + \ko(z_1,y_1)+\ko(z_2,y_2) + \ko(x_2,z_2) = \ko(x_1,y_1)+\ko(x_2,y_2) = \ko(x_1,\Omegal)+\ko(x_2,\Omegal). 
\]
The last estimate in combination with \eqref{key_ext_estimate} will prove the lemma. 

So we need to establish \eqref{otc_z}. Note that by \eqref{key_ext_estimate},
\begin{equation}\label{otc_u}
	\ko(u_1,u_2) \leq C' (1+ \jo(u_1,u_2))
\end{equation}
holds if $u_1,u_2\in \gamma_1$ or $u_1,u_2\in \gamma_2$. Then by Theorem A, 
\begin{equation*} 
	\dom(z_i)\geq \frac{\delta_\lambda}{b} \geq \frac{\delta_\lambda}{24(C')^2}.
\end{equation*}
Combining this with  $|z_1-z_2|\leq |x_1-z_1|+|x_2-z_2|+|x_1-x_2|< 3 \delta_\lambda$, we get 
\begin{equation} \label{otc_3}
	\jo(z_1,z_2) \leq \max_{i=1,2} \log \left(1+\frac{|z_1-z_2|}{\dom(z_i)} \right) \leq \log \left(1+\frac{3\cdot 24 (C')^2\delta_\lambda}{\delta_\lambda} \right) \leq 13 + 2\log(C').
\end{equation}

On the other hand, again due to \eqref{otc_u}, together with \eqref{G-P1} and \eqref{G-P2},
\[
\ko(z_i,y_i)\geq \jo(z_i,y_i)\geq \frac{1}{C'} \ko(z_i,y_i)-1\geq \frac{1}{C'}\log\frac{\lambda/4}{\dom(z_i)} -1 \geq  
\]
\[
\geq \frac{1}{C'}\log\frac{\lambda/4}{\dom(x_i) + |x_i-z_i|} -1 \geq \frac{1}{C'}\log\frac{\lambda}{12\delta_\lambda} -1. 
\]
Combining the last two estimates with \eqref{how_small_delta_is}, we get 
\[
\sum_{i = 1,2}\ko(z_i,\Omegal) = \sum_{i = 1,2}\ko(z_i,y_i) \geq 2\Big(\frac{1}{C'}\log\frac{\lambda}{12\delta_\lambda} -1\Big) \geq 2C'(14 + 2 \log C') > 2C'(\jo(z_1,z_2) + 1).
\]
Comparing the last estimate with \eqref{key_ext_estimate} applies to $z_1, z_2$, we obtain \eqref{otc_z}. 
\end{proof}

\section{Extensions of $\bmol$ functions on locally uniform domains}
\label{sec-extension}
Let $\Omega$ be an $\ed$-domain in $\Rn$ and $\Omega'$ be the interior of its complement. We will assume that neither set is empty. By $E$ and $E'$ we will denote the Whitney decompositions of $\Omega$ and $\Omega'$ respectively (see subsection \ref{sec_whitney_cubes} below for the definition). Given a mapping $Q \xrightarrow{} Q^*\in E$ defined for $Q$ in some sub-collection $\mathcal{E}' \subset E'$, we construct an extension operator 
\begin{equation} \label{ext_operator}
T_\lambda f(x) : = \begin{cases}
	f(x) & \text{ if } x\in \Omega,\\
	f_{Q^*} & \text{ if } x\in Q\in \mathcal{E}',\\
	0 & \text{ otherwise.}
\end{cases} 
\end{equation}
In this section we will prove the following theorem 
\begin{theorem} \label{thm_extension}
Let $\Omega$ be an $\ed$-domain and $\lambda\leq  \lambda_{\epsilon,\delta}:= \frac{\epsilon^2 \delta}{320 n (1+\sqrt{n}\epsilon)}$. Then there exists a mapping $Q \xrightarrow{} Q^*$ defined for all $Q\in E'$ with $l(Q)\leq \lambda$, such that the corresponding extension operator $T_\lambda$ is bounded from $\bmo_\lambda(\Omega)$ to $\bmo(\Rn)$. 
\end{theorem}

We will follow the choice of $Q\to Q^*$ introduced in the work of Jones \cite{jones2}. 

\begin{lem}[\cite{jones2} Lemma 2.4] 
	\label{lem_1_jones} 
	
	Let $\Omega$ be an $\ed$-domain. 
	If $Q\in E'$ and $\ell(Q)\leq \epsilon \delta/(16 n)$, then there exists $Q^*\in E$ such that 
	\begin{equation}
		\label{jones__first_matching_cond}
		1\leq \frac{\ell(Q^*)}{\ell(Q)} \leq 4
	\end{equation}
	and 
	\begin{equation}\label{jones__second_matching_cond}
		\dist(Q^*,Q)\leq C_{\epsilon,n} \ell(Q),
	\end{equation}
	where $C_{\epsilon,n} = 5\sqrt{n} + 8n \cdot \epsilon^{-2}$.
\end{lem}
In general, there may be several mappings satisfying \eqref{jones__first_matching_cond} and \eqref{jones__second_matching_cond}. We now fix any such correspondence $Q\to Q^*$ and will prove that the conclusions of Theorem \ref{thm_extension} hold.

\subsection{Whitney cubes of an $(\epsilon,\delta)$-domain }\label{sec_whitney_cubes}
The main goal of this section is to show that all points of an $\ed$-domain are ``close enough" to the ``large enough" Whitney cubes of the domain. The qualitative statement specifying what  ``close enough" and ``large enough" mean is the content of Proposition \ref{prop-minfat} below.

We recall that the Whitney decomposition of an open set $O\subsetneq \Rn$ is a collection of closed dyadic cubes $\{Q_I\}$ such that $O = \cup_i Q_i$ and 
\begin{equation}\label{WC1}
Q^\circ_i\cap Q_j^\circ = \varnothing \text{ if } i\neq j,
\end{equation}
\begin{equation}\label{WC2_dist_vs_size}
1\leq \frac{\dist(Q_j,\partial \Omega)}{\ell(Q_j)} \leq 4 \sqrt{n},
\end{equation}
\begin{equation}\label{WC3_adjacent_sides}
\frac{1}{4} \leq \frac{\ell(Q_j)}{\ell(Q_i)} \leq 4 \text{ if } Q_i\cap Q_j\neq \varnothing.
\end{equation}
Cubes in the Whitney decomposition of $O$ will be called {\em Whitney cubes} of $O$. Any two (closed) Whitney cubes $Q_1,Q_2$ such that 
$Q_1\cap Q_2 \neq 
\varnothing$ will be called {\em adjacent} cubes.  Depending on the context, this may include the case $Q_1 = Q_2$.  As in \cite{jones1}, we will use the term {\em Whitney chain} for a  finite sequence of distinct Whitney cubes with each cube adjacent to the preceding one.

\begin{lem}\label{lem_2}
Let $\Omega$ be an $\ed$-domain. If $Q$ is a dyadic cube in $\Rn$ with $\ell(Q)<\delta$, then there is $z\in Q^\circ$ such that 
$$
\dom(z) \geq \epsilon \ell(Q)/32. 
$$
\end{lem}

\begin{proof}
Let $c$ be the center of $Q$, $r=\ell(Q)/8$ and consider the open annulus 
$$A = \{x \in \Rn: r < |x - c| < 2r\} \subset Q^\circ.$$
If $A \cap \Omega = \varnothing$ , namely $A \subset \Omega'$, then $A$ contains points of $Q$ of distance at least $r/2 = \ell(Q)/16$ from $\bOmega$. Thus we may assume that there is an $x\in \Omega\cap A$. Applying the same reasoning to $A \setminus B(x, 2r)$,  which also contains an open ball of radius $r/2$, we may assume there is a $y\in \Omega\cap A$ such that 
$|x-y|>2r.$

Since $|x - y| \leq \diam(A) = 4r < \delta$, there is a curve $\gamma\subset \Omega$ connecting $x$ and $y$ and satisfying conditions  \eqref{q-convex_cond} and \eqref{John_cond}. Let $z$ be a point on this curve such that $|x-z|=r$. Then by the choice of $x, y$,
$$
|z-r|\leq |x-c|+|x-z| < 3r < \ell(Q)/2, 
$$
which means that $z\in Q^\circ$. Moreover, by $\eqref{John_cond}$ and since $|y - z| \geq |y - x| - |x - z| > r$,
$$
\dom(z) \geq \frac{\epsilon |x-z||y-z|}{|x-y|} \geq \frac{\epsilon r}{4} = \frac{\epsilon \ell(Q)}{32}.
$$	
\end{proof}

\begin{lem}
\label{lem3}
Let $\Omega$ be an $\ed$-domain.
If $Q\subset \Rn$ is a dyadic cube with $\ell(Q)<\delta$, 
then there is $Q_0\in E\cup E'$ with $Q_0\supset Q$ or $Q\supset Q_0$ and 
$$
\ell(Q_0)\ \geq \frac{\epsilon}{160 \sqrt{n}} \ell(Q). 
$$
\end{lem}

\begin{proof}
Apply Lemma \ref{lem_2} to the given cube $Q$ to get $z\in Q^\circ$ such that $\dom(z) \geq  \epsilon \ell(Q)/32$. Since $z \not\in \bOmega$, there is a cube $Q_0$  in $E\cup E'$ containing $z$. As $Q$ and $Q_0$ are dyadic cubes with intersecting interiors, either $Q_0\supset Q$ or $Q\supset Q_0$. There is nothing to prove in the former case.  

In the latter case, note that \eqref{WC2_dist_vs_size} implies 
$$
1 \leq \frac{\dom(z)}{\ell(Q_0)} \leq \frac{\dist(Q_0,\partial \Omega)+\diam(Q_0)}{\ell(Q_0)}\leq 5\sqrt{n}
$$
and therefore
$$
\ell(Q_0)  \geq \frac{\dom(z)}{5 \sqrt{n}}  \geq \frac{\epsilon  \ell(Q)}{160 \sqrt{n}}.
$$
\end{proof}

We can apply the lemma above to obtain the following result, which is reminiscent of the notion of {\em plumpness} in \cite{vaisala}.

\begin{prop}
\label{prop-minfat}
If $\Omega$ is an $\ed$-domain and $x \in \Omega$, then the Whitney decomposition $E$  contains a cube $S$ of sidelength at least $\epsilon \delta/(320 n)$, and whose distance from $x$ is less than  $\delta (\epsilon^{-1} + \sqrt{n})$.
\end{prop}

\begin{proof}
Let $x \in \Omega$ and take any dyadic cube $Q$ containing $x$ whose sidelength satisfies $\delta/2 \leq \ell(Q) < \delta$.  
By Lemma \ref{lem3}, there is a cube $Q_0\in E\cup E'$ such that $Q_0\supset Q$, or $Q\supset Q_0$ and $\ell(Q_0)\geq \epsilon\ell(Q)/(160 \sqrt{n})$. 
In the first case we must have $Q_0 \in E$ and $\dist(Q_0, x) = 0$, so we let $S = Q_0$.  

In the second case, it is possible that $Q_0 \in E'$, in which case, provided $\ell(Q_0) \leq \epsilon \delta/(16n)$, can apply Lemma~\ref{lem_1_jones} to $Q_0$ to get a matching cube $Q_0^* \in E$ with $\ell(Q_0^*) = \ell(Q_0)$, and let $S =  Q_0^*$.  If our original choice of $Q_0$ happens to be too large, going along the straight line from $Q_0 \subset \Omega'$ to $x \in \Omega$ which lies in $Q$, we obtain a chain of Whitney cubes in $E'$ whose sidelengths tend to zero as they approach $\bOmega$. By \eqref{WC3_adjacent_sides}, the sidelengths of consecutive cubes in the chain vary by a factor of at most $4$,  so we can replace our choice  by another cube $Q_0$ in $E'$ with sidelength in  $(\epsilon \delta/(64 n), \epsilon \delta/(16n))$, and take its matching cube as our desired cube $S$ in $E$.  Our $S$ will have $\ell(S) \geq \min(\epsilon\ell(Q)/(160 \sqrt{n}),\epsilon \delta/(64 n)) \geq \epsilon \delta/(320 n)$, and since our chosen cube $Q_0$ in $E'$ intersects our original cube $Q$, by \eqref{jones__second_matching_cond} we will have that 
$$\dist(S, x) \leq \dist(S, Q_0) + \diam(Q_0) + \dist(Q_0, x) \leq  (C_{\epsilon,n} + \sqrt{n})\ell(Q_0) + \diam(Q) < \delta (\epsilon^{-1} + \sqrt{n}).$$
\end{proof}

\subsection{Averages of $\BMO$ and $\bmo_\lambda$ functions over Whitney cubes}\label{sec_bmo_averages}
Let us start by stating the following special case of Lemma 2.2 in \cite{jones1}:
\begin{lem} 
\label{lem5_Jones}
Let $\Omega$ be any domain and $Q_1,Q_2\in E$ be two adjacent Whitney cubes of $\Omega$. Then  
\begin{equation*}
	|f_{Q_1} - f_{Q_2}| \lesssim \|f\|_{\BMOo},
\end{equation*}
for all $f\in\BMOo$.
\end{lem}

Our goal in this subsection is to give a good bound on $|f_{Q_1} - f_{Q_2}|$ for $f\in\BMOo$ and non-adjacent Whitney cubes $Q_1,Q_2$, using the $\ed$ condition of $\Omega$. This is done in Corollary \ref{lem6.5}.  Moreover, if $f\in \bmolo$, we will show that the averages $|f_{Q}|$ themselves grow at most logarithmically as the cubes shrink. While it is not difficult to see this working on $\Rn$, or on a metric measure spaces with doubling (see e.g. \cite{dafni_yue}), establishing the result in a domain requires the interplay between the geometry of the domain and the scale $\lambda$ (see Proposition \ref{prop-ave-bound}).

\begin{lem}
\label{lem6}
Let $\Omega$ be any domain and $x,y\in \Omega$ be connected by a rectifiable curve $\gamma$. Let $\{Q_i\}_{i=1}^m\subset E$ be the Whitney cubes covering $\gamma$. Then 
$$
m\lesssim \int_{\gamma} \frac{ds}{\dom(z)} +1. 
$$
In particular, if $\Omega$ is an $\ed$-domain, then for $x, y \in \Omega$ with $|x-y|< \delta$,  there exists a Whitney chain $\{Q_i\}_{i=1}^m$ such that $x\in Q_1$, $y\in Q_m$ and 
\begin{equation}
	\label{eq-GO}
	m \leq C_\epsilon (1+\jo(x,y)).
\end{equation}
\end{lem}

\begin{proof}
Fixing $\alpha \in (0,1)$, we have 
$$\int_{\gamma} \frac{ds}{\dom(z)} \geq \sum_i \frac{s(\gamma\cap Q_i)}{\dist(Q_i,\bOmega)} \geq  \sum_{i} \frac{s(\gamma\cap Q_i)}{\ell(Q_i)4\sqrt{n}}\gtrsim \; \#\{i: s(\gamma\cap Q_i) \geq \alpha \ell(Q_i)\}.$$
Let $\cA$ be the collection of those cubes $Q_i$ with $s(\gamma\cap Q_i) \geq \alpha \ell(Q_i)$.
For each $Q_i$, either one of the endpoints of $\gamma$ lies in an adjacent cube or $\gamma$ exits the set $\cN(Q_i)$ consisting of $Q_i$ and all its neighboring cubes, meaning that 
$$s(\gamma \cap \cN(Q_i)) \geq \dist(Q_i, \partial\cN(Q_i)) \geq \ell(Q_i)/4.$$
As the number of  cubes in $\cN(Q_i)$ is bounded by some $D_n$, and the size of the cubes in $\cN(Q_i)$ is bounded by $4\ell(Q_i)$, if all those cubes do not belong to $\cA$ then 
$$\ell(Q_i)/4 < 4D_n \alpha \ell(Q_i),$$
which is impossible if we choose $\alpha \leq 1/(16D_n)$.  Noting that there can be at most $2D_n$ cubes $Q_i$ for which $x$ or $y$ lie in an adjacent cube, we get that
$$m \leq D_n (\#\cA + 2)	\lesssim 	\int_{\gamma} \frac{ds}{\dom(z)} +1.$$
This proves the first part of the theorem. 

For the second part, we apply the above to the quasihyperbolic geodesic $\gamma'(x,y)$ and use Theorem \ref{thm_all_domains_equiv} with Theorem \ref{thm_GO-1_our version}.
\end{proof}

As a consequence of this lemma we obtain a local version of Lemma 2.1 in \cite{jones1}.
\begin{cor}
\label{lem6.5}
Assume $\Omega$ is an $\ed$-domain and let $Q_1,Q_2\in E$ be such that 
$\ell(Q_1) \leq \ell(Q_2)$. If $\dist(Q_1,Q_2)<\delta$, then 
$$
|f_{Q_1}-f_{Q_2}| \lesssim  \|f\|_{\BMO(\Omega)} \left(1 + \log_+\left(\frac{\dist(Q_1,Q_2) + \ell(Q_2)}{\ell(Q_1)}\right)\right)
$$
with a constant depending on $\epsilon$.  Here $\log_+(x) := \max(\log x, 0)$.
\end{cor}

\begin{proof}
We apply Lemma~\ref{lem6} to points $x \in Q_1$, $y \in Q_2$ with $|x - y| < \delta$ to get a curve between $x$ and $y$ for which we can use estimate \eqref{eq-GO}  to bound the length $m$ of the  chain of Whitney cubes along this curve by $C_\epsilon(\jo(x, y) + 1)$.  Note that $|x - y| \leq  \dist(Q_1,Q_2) + \diam(Q_1) + \diam(Q_2)$, so  the hypothesis and the properties of Whitney cubes, we have
$$\jo(x, y) = \frac 1 2 \log\left[\left(1+\frac{|x-y|}{\dist(x,\bOmega)}\right)\left(1+\frac{|x-y|}{\dist(y,\bOmega)}\right)\right] \lesssim 1 + \log_+\left(\frac{\dist(Q_1,Q_2)+ \ell(Q_2)}{\ell(Q_1)}\right),$$
The conclusion follows by applying Lemma~\ref{lem5_Jones} to adjacent cubes along this chain.	
\end{proof}

We now come to the desired logarithmic growth estimate on the averages of $\bmolo$ functions on Whitney cubes, which will prove very useful in what follows.
\begin{prop}
\label{prop-ave-bound}
Let $\Omega$ be an $\ed$-domain.   If $\lambda\leq \frac{\epsilon^2 \delta}{320 n(1 + \sqrt{n} \epsilon)}$, then for any $f\in \bmolo$ and $Q\in E$,
$$
|f_Q| \lesssim \left(1 + \log_+\left(\frac{\lambda}{\ell(Q)}\right)\right) \|f\|_{\bmolo}.
$$
\end{prop}

\begin{proof}
There is nothing to prove if $\ell(Q)\geq \lambda$, so we assume that $\ell(Q)\leq \lambda$. 

Given a chain of Whitney cubes $\{Q_i\}_{i=1}^m$ of length $m$, starting at $Q_1 = Q$, we can write
$$|f_Q| \leq \sum_{i=1}^{m-1} |f_{Q_i} - f_{Q_{i+1}}| + |f_{Q_m}|\leq m \|f\|_{\BMOo} + |f|_{Q_m}$$
due to Lemma \ref{lem5_Jones}.   Thus it remain to bound $m$ and $|f|_{Q_m}$.

Let $x \in Q$.  Since $\Omega$ is $\ed$, it is $(\epsilon, \delta')$ for $\delta' < \delta$.  Applying Proposition~\ref{prop-minfat} with $\delta'$ instead of $\delta$, where $\delta' \leq \delta  (\epsilon^{-1} + \sqrt{n})^{-1}$, we get the existence of a Whitney cube $S \in E$ whose sidelength is at least $\epsilon \delta'/(320 n)$, and whose distance from $x$ is less than $\delta' (\epsilon^{-1} + \sqrt{n}) \leq \delta$.  Taking a point $y \in S$ with $|x - y| < \delta$, we can use Lemma~\ref{lem6} to get the chain of cubes $\{Q_i\}_{i=1}^m$  with $Q_m = S$ and 
$$m \lesssim \log\left[\left(1+\frac{|x-y|}{\dist(x,\bOmega)}\right)\left(1+\frac{|x-y|}{\dist(y,\bOmega)}\right)\right] + 1 \lesssim \log\left[\left(1+\frac{|x - y|}{\ell(Q)}\right)\left(1+\frac{|x - y|}{\ell(Q_m)}\right)\right] +1,$$
by the properties of Whitney cubes.  Now since
$$\lambda \leq \frac{\epsilon^2 \delta}{320 n(1 + \sqrt{n} \epsilon)}  = \frac{\epsilon \delta}{320 n(\epsilon^{-1} + \sqrt{n})},$$ we can take $\delta' := 320n\epsilon^{-1} \lambda$, which satisfies the conditions above, and therefore
$$|x - y| < \delta' (\epsilon^{-1} + \sqrt{n})  = 320n  \epsilon^{-1} (\epsilon^{-1} + \sqrt{n}) \lambda = C_{n, \epsilon} \lambda.$$
Moreover, $\ell(Q_m) \geq \epsilon \delta'/(320 n) = \lambda$, 
giving us that $ |f|_{Q_m} \leq \|f\|_\bmolo$ and
$$m \lesssim  \log\left[\left(1+\frac{C_{n, \epsilon} \lambda}{\ell(Q)}\right)\left(1+\frac{C_{n, \epsilon} \lambda}{\lambda}\right)\right] +1\lesssim 1 + \log_+\left(\frac{\lambda}{\ell(Q)}\right).$$
\end{proof}

\subsection{Proof of Theorem~\ref{thm_extension}} \label{sec-ext}
Recall that we assume $\Omega$ to be an $\ed$-domain and $\lambda\leq \led$, where $\led:= \frac{\epsilon^2 \delta}{320 n(1 + \sqrt{n} \epsilon)}$.
For any $f\in \bmolo$, $\Tlam f:\Rn \to \R$ is defined by 
\begin{equation}\label{Ext_operator_T}
	\Tlam f (x) = \begin{cases}
		f(x) & \textup{if }  x\in \Omega; \\
		f_{Q^*} &\textup{if } x\in Q\in E': \ell(Q)\leq \lambda; \\
		0 & \textup{otherwise} .
	\end{cases}	
\end{equation}
where we fixed a mapping of $Q\to Q^*$, as in Lemma~\ref{lem_1_jones} defined for all $Q \in E'$ with $\ell(Q)\leq \lambda$. 

The following lemma, which is the $\bmo(\Rn)$ analogue of Lemma 2.3 in \cite{jones1}, allows us to measure the bmo norm only on dyadic cubes, provided we have control of the differences of averages over adjacent cubes.
\begin{lem}
	\label{lem8}
	Let $\cD(\Rn)$ be the collection of dyadic cubes in $\Rn$. Given $f\in \Loneloc(\Rn)$, denote
	$$
	a_f = \sup\limits_{Q\in \cD(\Rn)}\fint_Q |f(x) - f_Q| dx,
	$$
	$$
	b_f= \sup\limits_{\substack{Q_1,Q_2\in \cD(\Rn)\\\ell(Q_1)=\ell(Q_2)\\Q_1\cap Q_2 \neq \varnothing}} |f_{Q_1}-f_{Q_2}|,
	$$
	and
	$$
	c_f=\sup_{Q\in \cD(\Rn): \ell(Q)\geq\lambda/16\sqrt{n}} |f|_Q. 
	$$
	Then 
	$$\|f\|_{\bmol(\Rn)} \lesssim a_f +b_f+c_f.
	$$
\end{lem}

\begin{proof}
	The inequality
	$$
	\sup_{Q\subset \Omega} \fint_Q |f(x)-f_Q| dx \lesssim a_f+b_f
	$$
	is Lemma 2.3 in \cite{jones1}. More precisely, its proof shows that if $Q\subset \Omega$, $\{Q_i\}$ is the Whitney decomposition of the interior of $Q$ and $Q_0$ is the cube in this decomposition which contains the center of $Q$, then 
	$$
	\fint_Q |f(x)-f_{Q_0}| dx \lesssim a_f + b_f.
	$$
	When $\ell(Q)\geq \lambda$, the sidelength of this chosen cube satisfies $\ell(Q_0)\geq \lambda/(16\sqrt{n})$ and 
	$$|f|_Q \leq \fint_Q |f(x)-f_{Q_0}| dx + |f|_{Q_0} \lesssim  a_f +b_f+c_f.
	$$
\end{proof}	

The last lemma shows that we will prove Theorem \ref{thm_extension} once we establish that for some $C_\epsilon>0$, 
\[
\sup\limits_{\substack{Q\in \cD(\Rn)\\
		\ell(Q)\leq \lambda}} \fint_Q |\Tlam(x) - (\Tlam f)_Q| dx \leq C_\epsilon  \|f\|_{\bmolo},
\]
\[
\sup\limits_{\substack{Q_1,Q_2\in \cD(\Rn) \\\ell(Q_1)=\ell(Q_2)\\Q_1,Q_2 \text{ are adjacent}}} |(\Tlam f)_{Q_1}-(\Tlam f)_{Q_2}| \leq C_\epsilon \|f\|_{\bmolo}
\]
and 
\[
\sup\limits_{\substack{Q\in \cD(\Rn)\\
		\ell(Q)\geq \lambda/16\sqrt{n}}} |\Tlam f|_Q \leq C_\epsilon \|f\|_{\bmolo}.
\]

This is done in Lemma~\ref{lem10} below, using Lemmas~\ref{lem8.5} and \ref{lem9}.

\begin{lem}
\label{lem8.5}
Let $Q\in E \cup E'$.  Then 
$$|(\Tlam f)_{Q}| \lesssim \left(1 + \log_+\left( \frac{\lambda}{\ell(Q)}\right)\right) \|f\|_{\bmolo}.
$$
Moreover, if $Q_1, Q_2 \in E \cup E'$ with $\ell(Q_1) \leq \ell(Q_2)$ and  $\dist(Q_1,Q_2)<\delta$, then 
$$
|(\Tlam f)_{Q_1}-(\Tlam f)_{Q_2}| \lesssim \left(1 + \log_+\left(\frac{\min(\dist(Q_1,Q_2)+ \ell(Q_2), \lambda)}{\ell(Q_1)}\right)\right) \|f\|_{\bmolo}.
$$
The constants in both inequalities depend on $\epsilon$.
\end{lem}

\begin{proof}
The first inequality follows from the definition of $\Tlam$, Proposition~\ref{prop-ave-bound} and Lemma~\ref{lem_1_jones}, noting that the matching cube $Q^*$ satisfies $\ell(Q^*) \geq \ell(Q)$.

To establish the second inequality, we may thus assume $\dist(Q_1,Q_2) + \ell(Q_2) < \lambda$. Then both cubes have sidelength bounded by $\lambda$, so by the definition of $\Tlam$ and Lemma~\ref{lem_1_jones} we can write
$$|(\Tlam f)_{Q_1}-(\Tlam f)_{Q_2}|  = |f_{S_1} - f_{S_2}|$$
for $S_1, S_2 \in E$ with $\frac{1}{4} \leq \frac{\ell(S_i)}{\ell(Q_i)} \leq 4$ (possibly $S_i = Q_i$) and $\dist(S_i,Q_i)\leq C_{\epsilon,n} \ell(Q_i)$, where $C_{\epsilon,n} = 5\sqrt{n} + 8n \cdot \epsilon^{-2}$.  Thus
\begin{eqnarray*}
	\dist(S_1, S_2) 
	& \leq &  C_{\epsilon,n}\ell(Q_1) + \diam(Q_1) + \dist(Q_1, Q_2) + \diam(Q_2) +  C_{\epsilon,n}\ell(Q_2)\\
	& \leq & 2(C_{\epsilon,n} + \sqrt{n}) \ell(Q_2) + \dist(Q_1, Q_2)\\
	& < & 2(6\sqrt{n} + 8n \cdot \epsilon^{-2})\lambda \\
	& \leq & 2(6\sqrt{n} + 8n \cdot \epsilon^{-2})\frac{\epsilon^2 \delta}{320 n(1 + \sqrt{n} \epsilon)} < \delta,
\end{eqnarray*}
and the result follows from Corollary~\ref{lem6.5}.
\end{proof}

\begin{lem}
\label{lem9}
Let $Q$ be a dyadic cube in $\Rn$ with $\ell(Q)<\delta$.  Assume $Q \not\subset \Omega$  and $Q$ is not contained inside any Whitney cube in $E'$.  
Then there exists a  Whitney cube  $Q_0\in E\cup E'$ such that $Q_0 \subset Q$, $|Q_0| \gtrsim \epsilon^n |Q|$ and
$$
\fint_Q |\Tlam f(x)-(\Tlam f)_{Q_0}| \leq C_\epsilon \|f\|_{\bmolo}.
$$
\end{lem}

\begin{proof}
We follow the proof of Lemma 2.11 in \cite{jones1}.
By Lemma \ref{lem3}, there is a cube $Q_0\in E\cup E'$ with $Q_0\supset Q$ or $Q\supset Q_0$ and $\ell(Q_0)\geq \frac{\epsilon}{160\sqrt{n}} \ell(Q)$.  The assumptions on $Q$ rule out the first case, so  we have $2^{-k} \ell(Q) \leq \ell(Q_0) < \ell(Q)$, where $k$ is the positive integer for which $2^{-k} \leq \frac{\epsilon}{160\sqrt{n}} < 2^{-k+1}$.

Partition $Q$ into $2^{kn}$ dyadic cubes $\{Q^1_j\}$ of sidelength $2^{-k} \ell(Q)$.  If there exists some Whitney cube $S^1_j \in E\cup E'$ with $Q^1_j \subset S^1_j$, we say that $Q^1_j$ belongs to $F_1$. Since at least one of the $Q^1_j$  is contained in $Q_0$,  $F_1 \neq \varnothing$,  and therefore, denoting by $R_1$ the union of all  subcubes not in $F_1$, we have  
$$|R^1| = |Q| -\sum_{Q^1_j  \in F_1}|Q^1_j | \leq  (1 - 2^{-kn})|Q| \leq (1-\epsilon_1^n)|Q|,$$
where we set $\epsilon_1:= \frac{\epsilon}{320\sqrt{n}}$.     

Partitioning each of cubes $Q^1_j \not\in F_1$ further  into  $2^{kn}$ dyadic cubes of sidelength $2^{-k} \ell(Q^1_j) = 2^{-2k}\ell(Q)$, we denote all of those cubes by $\{Q^2_{j'}\}$.
By Lemma \ref{lem3} applied to  $Q^1_j \not\in F_1$, we again have that at least one of the subcubes $Q^2_{j'}$ is contained in a Whitney cube $S^2_{j'} \in E\cup E'$, and furthermore $S^2_{j'} \subset Q^1_j$.   We collect all the subcubes  $Q^2_{j'}$ which lie in Whitney cubes into a collection $F_2$, and denote the union of the remaining subcubes by $R_2$.
Then again
$$
|R^2|\leq (1-\epsilon_1^n)^2 |Q|. 
$$
We continue this process recursively.  For each $N \in \N$, at the $N$th stage we have 
$$Q = \bigcup_{i = 1}^N\bigcup_{Q^i_j  \in F_i}Q^i_j \cup R_N,$$
where the union is over a finite collection of pairwise disjoint cubes and the remainder set satisfies 
\begin{equation}
	\label{eq-remainder}
	|R_N| \leq  (1-\epsilon_1^n)^N |Q|.
\end{equation}
Letting $N \ra 0$, we have that $\chi_Q = \displaystyle{\sum_{i = 1}^\infty\sum_{Q^i_j  \in F_i}\chi_{Q^i_j}}$ almost everywhere and by the monotone convergence theorem

\begin{eqnarray*}
	\fint_Q |\Tlam f(x) -(\Tlam f)_{Q_0}| dx & = & \sum_{i = 1}^\infty\sum_{Q^i_j  \in F_i} \fint_Q |\Tlam f(x) -(\Tlam f)_{Q_0}| dx\\
	& \leq & 
	\sum_{i = 1}^\infty\sum_{Q^i_j  \in F_i}\frac{|Q^i_j|}{|Q|} \left( \fint_{Q^i_j} |\Tlam f(x)-(\Tlam f)_{Q^i_j}| dx + |(\Tlam f)_{Q^i_j}-(\Tlam f)_{Q_0}| \right).
\end{eqnarray*}
Since each $Q^i_j$ is contained in a Whitney cube in $E\cup E'$, either $Q^i_j \subset \Omega$ or $\Tlam f$ is constant on  $Q^i_j$, so
\begin{equation}
	\label{eq-Qij}
	\fint_{Q^i_j} |\Tlam f(x)-(\Tlam f)_{Q^i_j}| dx  \leq  \|f\|_\BMOo \leq 2\|f\|_{\bmolo}.
\end{equation}
Moreover, by the selection of $Q^i_j$, we know that the Whitney cube containing $Q^i_j$, denoted by $S^i_j$, must have $\ell(S^i_j) \leq 2^{-k(i-1)}\ell(Q) = 2^k\ell(Q^i_j)$.  If $S^i_j \in E'$ then $\Tlam f$ is constant on $S^i_j$ and $(\Tlam f)_{S^i_j} =  (\Tlam f)_{Q^i_j}$, while if $S^i_j \in E$ then 
$$|(\Tlam f)_{Q^i_j} - (\Tlam f)_{S^i_j}| \leq \fint_{Q^i_j} |\Tlam f(x)-(\Tlam f)_{S^i_j}| dx \leq 2^{kn}\fint_{S^i_j} |\Tlam f(x)-(\Tlam f)_{S^i_j}| dx \leq C_\epsilon\|f\|_{\bmolo}.$$
Recalling that $Q_0 \subset Q$ with $Q_0 \in E \cap E'$   and $\ell(Q_0) \geq 2^{-k}\ell(Q) \geq \ell(Q^i_j)$, we  apply Lemma~\ref{lem8.5} to $S^i_j$ and $Q_0$ to get
\begin{eqnarray*}
	|(\Tlam f)_{Q^i_j}-(\Tlam f)_{Q_0}| & \leq & C_\epsilon \|f\|_{\bmolo} + |(\Tlam f)_{S^i_j} - (\Tlam f)_{Q_0}|\\
	& \lesssim &  \|f\|_{\bmolo} +
	\left(1 + \log_+\left(\frac{\dist(S^i_j,Q_0) + \ell(Q_0)}{\ell(S^i_j)}\right)\right) \\
	&\lesssim & \|f\|_{\bmolo} + \log_+\left(\frac{\diam(Q)}{2^{-k(i-1)}\ell(Q)}\right)\leq C_{\epsilon} \; i \|f\|_{\bmolo}
\end{eqnarray*}
Finally, combining the previous estimate with \eqref{eq-Qij}, noting that 
$\sum_j |Q^i_j| \leq |R_{i-1}|$,
and using  \eqref{eq-remainder}, we have
$$
\fint_Q |\Tlam f(x) -(\Tlam f)_{Q_0}| dx \lesssim \sum_{i = 1}^\infty \sum_{Q^i_j  \in F_i}\frac{|Q^i_j|}{|Q|}i  \|f\|_{\bmolo}  \lesssim \sum_{i = 1}^\infty(1-\epsilon_1^n)^{i-1}i  \|f\|_{\bmolo} \lesssim  \|f\|_{\bmolo}
$$
with a constant depending on $\epsilon$.
\end{proof}

\begin{lem}
\label{lem10}
There are constants depending on $\epsilon$ such that
$$a_f = \sup\limits_{\substack{Q\in \cD(\Rn)\\
		\ell(Q)\leq \lambda}} \fint_Q |\Tlam(x) - (\Tlam f)_Q| dx \lesssim  \|f\|_{\bmolo},
$$
$$b_f= \sup\limits_{\substack{Q_1,Q_2\in \cD(\Rn) \\\ell(Q_1)=\ell(Q_2)\\Q_1,Q_2 \text{ are adjacent}}} |(\Tlam f)_{Q_1}-(\Tlam f)_{Q_2}| \leq \|f\|_{\bmolo}$$
and 
$$c_f = \sup\limits_{\substack{Q\in \cD(\Rn)\\
		\ell(Q)\geq \lambda/16\sqrt{n}}} |\Tlam f|_Q \lesssim \|f\|_{\bmolo}.
$$
\end{lem}

\begin{proof}
Let us first consider a dyadic cube $Q\subset \Rn$ with $\ell(Q)\leq \lambda$. If $Q \subset \Omega$ or $Q \subset S \in E'$ then, by definition, the oscillation of $\Tlam$ on $Q$ is bounded by $\|f\|_\bmolo$ or zero.  Otherwise, as $\ell(Q)<\delta$, we have a Whitney subcube $Q_0 \in E\cup E'$ for which Lemma~\ref{lem9} gives 
\begin{equation}
	\label{lem10_1}
	\fint_Q |\Tlam f(x) - (\Tlam f)_Q| dx \leq 2\fint_Q |\Tlam f(x) - (\Tlam f)_{Q_0}| dx \lesssim \|f\|_{\bmolo}.
\end{equation}
If in addition  $\ell(Q) \geq \frac{\lambda}{16\sqrt{n}}$, then we also have, by applying Lemma~\ref{lem8.5} to $Q_0$,
$$|Tf_\lambda|_Q \leq \fint_{Q} |\Tlam f(x) - (\Tlam f)_{Q_0}| dx +  |(\Tlam f)_{Q_{0}}| \lesssim (1 + \log_+(16\sqrt{n})) \|f\|_{\bmolo}. $$

If $Q\subset \Rn$ is dyadic with $\ell(Q) > \lambda$, then  we can partition $Q$ into $2^k$ equal dyadic subcubes ${Q_i}$ of sidelength $2^{-k} \ell(Q)$, where $k\in \N$ is such that $2^{-k} \leq \lambda/\ell(Q)< 2^{1-k}$, and write the mean $|Tf_\lambda|_Q$ as the average of the means taken over the $Q_i$.
Since   $\ell(Q_i) \leq \lambda < \delta$,
we can apply Lemma~\ref{lem9} to each of the $Q_i$, and denote by $Q_{i,0}$ the corresponding Whitney cubes.  Then applying Lemma~\ref{lem8.5} to the 
$Q_{i,0}$, and using the fact that $\ell(Q_i)=2^{-k} \ell(Q) \geq \lambda/2$, we have
\begin{eqnarray*}
	|Tf_\lambda|_Q& \leq & \sum_{i} \frac{|Q_i|}{|Q|} \cdot \left( \fint_{Q_i} |\Tlam f(x) - (\Tlam f)_{Q_{i,0}}| dx +  |(\Tlam f)_{Q_{i,0}}| \right) \\
	& \lesssim & \|f\|_{\bmolo} +  \left(1 + \log_+\left( \frac{\lambda}{2^{-k}\ell(Q)}\right)\right) \|f\|_{\bmolo} \lesssim  \|f\|_{\bmolo}.
\end{eqnarray*}

Finally, consider adjacent dyadic cubes $Q_1,Q_2\subset \Rn$ of equal sidelength.  If this sidelength is at least $\lambda$, then we can apply the previous estimate to get
$$
|\Tlam f_{Q_1} - (\Tlam f)_{Q_2}| \lesssim \|f\|_{\bmolo}.
$$
Hence, we assume that  $\ell(Q_1)=\ell(Q_2) < \lambda$ and again, as $\lambda < \delta$, apply Lemma \ref{lem9} to get Whitney cubes $Q_{1,0}\subset Q_1$ and $Q_{2,0}\subset Q_2$ such that 
$$
|(\Tlam f)_{Q_i} -(\Tlam f)_{Q_{i,0}}| \leq \fint_{Q_i} |\Tlam f(x)-(\Tlam f)_{Q_{i,0}}| \lesssim \|f\|_{\bmolo}, \quad  i=1,2.
$$
Moreover, by Lemma~\ref{lem8.5},
$$
|(\Tlam f)_{Q_{1,0}}-(\Tlam f)_{Q_{2,0}}| \lesssim \left(1 + \log_+\left(\frac{\dist(Q_{1,0},Q_{2,0})}{\ell(Q_{1,0})} + 1\right)\right) \|f\|_{\bmolo} \lesssim  \|f\|_{\bmolo}
$$
since $\dist(Q_{1,0},Q_{2,0}) \leq \dist(Q_{1},Q_{2}) + \diam(Q_1) + \diam(Q_2)  = 2\sqrt{n}\ell(Q_1)$ and $\ell(Q_{1,0}) \gtrsim \epsilon \ell(Q_1)$.
Hence, 
$$
|(\Tlam f)_{Q_1} - (\Tlam f)_{Q_2}| \leq \sum_{i = 1,2}|(\Tlam f)_{Q_i} -(\Tlam f)_{Q_{i,0}}| + |(\Tlam f)_{Q_{1,0}} - (\Tlam f)_{Q_{2,0}}|  \lesssim  \|f\|_{\bmolo}.
$$
\end{proof}

\providecommand{\bysame}{\leavevmode\hbox to3em{\hrulefill}\thinspace}
\providecommand{\MR}{\relax\ifhmode\unskip\space\fi MR }
\providecommand{\MRhref}[2]{%
\href{http://www.ams.org/mathscinet-getitem?mr=#1}{#2}
}
\providecommand{\href}[2]{#2}

\end{document}